\documentclass[a4paper,11pt]{article}
\usepackage{amsmath,amsthm,amssymb,enumitem,xcolor}
\usepackage{tikz}
\usetikzlibrary{calc,decorations,decorations.pathreplacing}

\usepackage[hang]{footmisc}
\usepackage[nosort,nocompress,noadjust]{cite}

\usepackage[bookmarks=false,hyperfootnotes=false,colorlinks,
    linkcolor={red!60!black},
    citecolor={blue!50!black},
    urlcolor={blue!80!black}]{hyperref}

\renewcommand{\eqref}[1]{\hyperref[#1]{(\ref{#1})}}

\newlist{enumlist}{enumerate}{1}
\setlist[enumlist]{labelindent=0cm,label=(\roman*),labelwidth=4.5ex,labelsep=0.5ex,leftmargin=5ex,align=left,topsep=0.5ex,itemsep=1ex,parsep=1ex}

\newlist{itemlist}{itemize}{1}
\setlist[itemlist]{labelindent=0cm,label=$\bullet$,labelwidth=3.5ex,labelsep=0.5ex,leftmargin=4ex,align=left,topsep=0.5ex,itemsep=1ex,parsep=1ex}

\numberwithin{equation}{section}

{\theoremstyle{definition}\newtheorem{definition}{Definition}[section]

\newtheorem{remark}[definition]{Remark}
}

\newtheorem{proposition}[definition]{Proposition}
\newtheorem{lemma}[definition]{Lemma}
\newtheorem{theorem}[definition]{Theorem}
\newtheorem{corollary}[definition]{Corollary}

\renewcommand{\Re}{\operatorname{Re}}

\providecommand{\cZ}{\mathcal{Z}}
\providecommand{\cL}{\mathcal{L}}
\providecommand{\TLJ}{\operatorname{TLJ}}
\providecommand{\Gammahat}{\widehat{\Gamma}}
\providecommand{\bF}{\mathbb{F}}
\providecommand{\bFhat}{\widehat{\mathbb{F}}}
\providecommand{\Z}{\mathbb{Z}}
\providecommand{\GL}{\operatorname{GL}}
\providecommand{\cKfix}{\mathcal{K}^{\text{\rm fix}}}
\providecommand{\cB}{\mathcal{B}}
\providecommand{\be}{\beta}
\providecommand{\cG}{\mathcal{G}}
\providecommand{\ubar}{\overline{u}}
\providecommand{\N}{\mathbb{N}}
\providecommand{\SU}{\operatorname{SU}}
\providecommand{\bG}{\mathbb{G}}
\providecommand{\bH}{\mathbb{H}}
\providecommand{\bGhat}{\widehat{\mathbb{G}}}
\providecommand{\bHhat}{\widehat{\mathbb{H}}}
\providecommand{\otmin}{\otimes_{\text{\rm min}}}
\providecommand{\Ubar}{\overline{U}}
\providecommand{\om}{\omega}
\providecommand{\albar}{\overline{\alpha}}
\providecommand{\cV}{\mathcal{V}}
\providecommand{\cA}{\mathcal{A}}
\providecommand{\cAtil}{\widetilde{\mathcal{A}}}
\providecommand{\cK}{\mathcal{K}}

\providecommand{\eps}{\varepsilon}
\providecommand{\cE}{\mathcal{E}}
\providecommand{\ot}{\otimes}
\providecommand{\ovt}{\mathbin{\overline{\otimes}}}
\providecommand{\al}{\alpha}
\providecommand{\C}{\mathbb{C}}
\providecommand{\recht}{\rightarrow}
\providecommand{\cC}{\mathcal{C}}
\providecommand{\RR}{\mathbb{R}}
\providecommand{\NN}{\mathbb{N}}
\providecommand{\ZZ}{\mathbb{Z}}

\providecommand{\CC}{\mathbb{C}}
\providecommand{\onb}{\,\operatorname{onb}}
\providecommand{\counit}{\varrho}
\mathchardef\mhyphen="2D
\providecommand{\ind}[1]{\mathrm{ind\mhyphen}#1}
\providecommand{\Tr}{\operatorname{Tr}}

\providecommand{\Hom}{\operatorname{Hom}}
\providecommand{\Irr}{\operatorname{Irr}}
\providecommand{\mult}{\operatorname{mult}}

\providecommand{\id}{\operatorname{id}}
\providecommand{\Rep}{\operatorname{Rep}}

\providecommand{\Ext}{\operatorname{Ext}}
\providecommand{\Tor}{\operatorname{Tor}}
\providecommand{\Pol}{\operatorname{Pol}}
\providecommand{\CatKer}{\operatorname{Ker}}

\makeatletter
\providecommand*{\diff}%
        {\@ifnextchar^{\DIfF}{\DIfF^{}}}
\def\DIfF^#1{%
        \mathop{\mathrm{\mathstrut d}}%
                \nolimits^{#1}\gobblespace
}
\def\gobblespace{%
        \futurelet\diffarg\opspace}
\def\opspace{%
        \let\DiffSpace\!%
        \ifx\diffarg(%
                \let\DiffSpace\relax
        \else
                \ifx\diffarg\[%
                        \let\DiffSpace\relax
                \else
                        \ifx\diffarg\{%
                                \let\DiffSpace\relax
                        \fi\fi\fi\DiffSpace} %

\makeatother

\newcommand{\roundedbox}[4]{
	\draw[rounded corners=5pt, thick, fill=white] ($#1+(-#2,-#2)+(-#3,0)$) rectangle ($#1+(#2,#2)+(#3,0)$);
	\coordinate (ZZa) at ($#1+(-#3,0)$);
	\coordinate (ZZb) at ($#1+(#3,0)$);
	\node at ($1/2*(ZZa)+1/2*(ZZb)$) {#4};
}

\begin{document}

\begin{center}
{\boldmath\LARGE\bf $L^2$-Betti numbers of rigid $C^*$-tensor categories and \vspace{0.5ex}\\ discrete quantum groups}

\bigskip

{\sc by David Kyed{\renewcommand\thefootnote{1}\footnote{\noindent Department of Mathematics and Computer Science, University of Southern Denmark, Odense (Denmark).\\ E-mail: dkyed@imada.sdu.dk. DK is supported by the Villum foundation grant 7423.}}, Sven Raum{\renewcommand\thefootnote{2}\footnote{\noindent EPFL SB SMA, Lausanne (Switzerland). E-mail: sven.raum@epfl.ch.}}, Stefaan Vaes{\renewcommand\thefootnote{3}\footnote{\noindent KU~Leuven, Department of Mathematics, Leuven (Belgium). E-mails: stefaan.vaes@kuleuven.be and matthias.valvekens@kuleuven.be. Supported by European Research Council Consolidator Grant 614195, and by long term structural funding~-- Methusalem grant of the Flemish Government.}} and Matthias Valvekens$^3$}
%
%
%
\end{center}

\begin{abstract}
\noindent We compute the $L^2$-Betti numbers of the free $C^*$-tensor categories, which are the representation categories of the universal unitary quantum groups $A_u(F)$. We show that the $L^2$-Betti numbers of the dual of a compact quantum group $\bG$ are equal to the $L^2$-Betti numbers of the representation category $\Rep(\bG)$ and thus, in particular, invariant under monoidal equivalence. As an application, we obtain several new computations of $L^2$-Betti numbers for discrete quantum groups, including the quantum permutation groups and the free wreath product groups. Finally, we obtain upper bounds for the first $L^2$-Betti number in terms of a generating set of a $C^*$-tensor category.
\end{abstract}

\section{Introduction}

The framework of rigid $C^*$-tensor categories unifies a number of structures encoding various kinds of quantum symmetry, including standard invariants of Jones' subfactors, representation categories of compact quantum groups, in particular of $q$-deformations of compact simple Lie groups, and ordinary discrete groups. In several respects, rigid $C^*$-tensor categories are quantum analogues of discrete groups.

Using this point of view, the unitary representation theory for rigid $C^*$-tensor categories was introduced in \cite{pv-repr-subfactors}. This allowed to define typical geometric group theory properties like the Haagerup property and property~(T) intrinsically for standard invariants of subfactors and for rigid $C^*$-tensor categories. It was then proved in \cite{pv-repr-subfactors}, using \cite{arano,de-commer-freslon-yamashita}, that the Temperley-Lieb-Jones category $\Rep(\SU_q(2))$ has the Haagerup property, while $\Rep(\SU_q(3))$ has Kazhdan's property~(T). Equivalent formulations of the unitary representation theory of a rigid $C^*$-tensor category were found in \cite{neshveyev-yamashita,ghosh-jones} and are introduced below.

In \cite{psv-cohom}, a comprehensive (co)homology theory for standard invariants of subfactors and rigid $C^*$-tensor categories was introduced. Taking the appropriate Murray-von Neumann dimension for (co)homology with $L^2$-coefficients, this provides a definition of $L^2$-Betti numbers.

The first goal of this article is to compute the $L^2$-Betti numbers for the representation category $\cC$ of a free unitary quantum group $A_u(F)$. Here, $A_u(F)$ is the universal compact quantum group (in the sense of Woronowicz) generated by a single irreducible unitary representation. As a $C^*$-tensor category, $\cC$ is the free rigid $C^*$-tensor category generated by a single irreducible object $u$. The irreducible objects of $\cC$ are then labeled by all words in $u$ and $\ubar$ and can thus be identified with the free monoid $\NN * \NN$. We prove that $\beta_1^{(2)}(\cC)=1$ and that the other $L^2$-Betti numbers vanish.

For compact quantum groups $\bG$ of Kac type (a unimodularity assumption that is equivalent with the traciality of the Haar state), the $L^2$-Betti numbers $\beta_n^{(2)}(\bGhat)$ of the dual discrete quantum group $\bGhat$ were defined in \cite{kyed-quantum-l2hom}. The second main result of our paper is that these $L^2$-Betti numbers only depend on the representation category of $\bG$ and are given by $\beta_n^{(2)}(\Rep(\bG))$. This is surprising for two reasons. The $L^2$-Betti numbers $\beta_n^{(2)}(\Rep(\bG))$ are well defined for all compact quantum groups, without unimodularity assumption. And secondly, taking arbitrary coefficients instead of $L^2$-cohomology, there is no possible identification between the (co)homology of $\bGhat$ and $\Rep(\bG)$. Indeed, by \cite[Theorem 3.2]{collins-haertel-thom}, homology with trivial coefficients distinguishes between the quantum groups $A_o(k)$, but does not distinguish between their representation categories $\Rep(A_o(k))$ by Corollary \ref{cor.homology-trivial-coefficients} below. As an application, we compute the $L^2$-Betti numbers for several families of Kac type discrete quantum groups, including the duals of the quantum permutation groups $S_m^+$, the hyperoctahedral series $H_m^{s+}$ of \cite{banica-verngioux-reflection} and the free wreath product groups $\bH \wr_* \bF$ of \cite{bichon-wreath}.

One of the equivalent definitions in \cite{psv-cohom} for the (co)homology of a rigid $C^*$-tensor category $\cC$ is given by the Hochschild (co)homology of the associated tube algebra $\cA$ together with its counit $\counit : \cA \recht \C$ as the augmentation. In \cite{neshveyev-yamashita-remarks}, it is proved that when $\cC = \Rep(\bG)$ is the representation category of a compact quantum group $\bG$, then the tube algebra $\cA$ is strongly Morita equivalent with the Drinfeld double algebra of $\bG$. This is one of the main tools in our paper. As a side result, applying this to $\bG = \SU_q(2)$, so that $\cC$ becomes the Temperley-Lieb-Jones category TLJ, we can transfer the resolution of \cite{bichon-yetter-drinfeld} to a length 3 resolution for the tube algebra of TLJ, see Theorem \ref{thm.resolution-TLJ}. This allows us in particular to compute the (co)homology of TLJ with trivial coefficients, giving $\C$ in degree $0$ and degree $3$, and giving $0$ in all other degrees. This completes the computation in \cite[Proposition 9.13]{psv-cohom}, which went up to degree $2$, and this was also obtained in an unpublished note of Y.~Arano.

In the second part of this paper, we focus on the first $L^2$-Betti number of a rigid $C^*$-tensor category. For an infinite group $\Gamma$ generated by $n$ elements $g_1,\ldots,g_n$, it is well known that $\beta_1^{(2)}(\Gamma) \leq n-1$. The reason for this is that a $1$-cocycle on $\Gamma$ is completely determined by the values it takes on the generators $g_1,\ldots,g_n$. In Section \ref{sec.derivations-Cstar-tensor}, we explain how to realize the first cohomology of a rigid $C^*$-tensor category $\cC$ by a kind of derivations $D$ and prove that $D$ is indeed determined by its values on a generating set of irreducible objects. We then deduce an upper bound for $\beta_1^{(2)}(\cC)$ and show in Section \ref{sec.derivations-Au} that this upper bound is precisely reached for the universal (or free) category $\cC = \Rep(A_u(F))$.

{\bf Acknowledgment.} SV would like to thank the Isaac Newton Institute for Mathematical Sciences for support and hospitality during the programme {\it Operator Algebras: Subfactors and their Applications} when work on this paper was undertaken, supported by EPSRC Grant Number EP/K032208/1. SV also thanks Dimitri Shlyakhtenko for several helpful remarks.

\section{Preliminaries}

\subsection{\boldmath The tube algebra of a rigid $C^*$-tensor category}

Let $\mathcal{C}$ be a \textit{rigid $C^*$-tensor category}, i.e.\@ a $C^*$-tensor category with irreducible unit object $\eps\in\mathcal{C}$ such that every object $\alpha\in\mathcal{C}$ has a conjugate $\overline{\alpha}\in\mathcal{C}$. In particular, this implies that every object in $\mathcal{C}$ decomposes into finitely many irreducibles. The essential results on rigid $C^*$-tensor categories, which we will use without further reference, are covered in \cite[Chapter 2]{neshveyev-tuset}. For $\alpha,\beta\in\mathcal{C}$, we denote the (necessarily finite-dimensional) Banach space of morphisms $\alpha\to\beta$ by $(\beta,\alpha)$.

The set of isomorphism classes of irreducible objects of $\mathcal{C}$ will be denoted by $\Irr(\mathcal{C})$. In what follows, we do not distinguish between irreducible objects and their respective isomorphism classes and we fix representatives for all isomorphism classes once and for all. Additionally, we always identify $(\alpha,\alpha)$ with $\CC$ when $\alpha\in\Irr(\mathcal{C})$.
The \textit{multiplicity} of $\gamma$ in $\alpha$ when $\alpha\in\mathcal{C}$ and $\gamma\in\Irr(\mathcal{C})$ is defined by
\[
\mult(\gamma,\alpha)=\dim_{\CC} (\alpha,\gamma) \; .
\]
For $\alpha,\beta\in\mathcal{C}$, we write $\beta\prec\alpha$ whenever $\beta$ is isomorphic with a subobject of $\alpha$.
When there is no danger of confusion, we denote the tensor product of $\alpha$ and $\beta$ by $\alpha\beta$.

The rigidity assumption says that every object $\alpha\in\mathcal{C}$ admits a \textit{solution to the conjugate equations} \cite[\S~2.2]{neshveyev-tuset}, i.e.\@ an object $\overline{\alpha}\in\mathcal{C}$ and a pair of morphisms $s_\alpha\in(\alpha\overline{\alpha},\eps)$ and $t_\alpha\in(\overline{\alpha}\alpha,\eps)$ satisfying the relations
\[
(t_\alpha^*\otimes 1)(1 \otimes s_\alpha)=1 \text{\qquad and\qquad}(s_\alpha^*\otimes 1)(1 \otimes t_\alpha)=1 \; .
\]
A \textit{standard} solution for the conjugate equations for $\alpha\in\mathcal{C}$ additionally satisfies
\[
s_\alpha^*(T\otimes 1)s_\alpha =t_\alpha^*(1 \otimes T)t_\alpha
\]
for all $T\in (\alpha,\alpha)$.
The adjoint object $\albar$ and the standard solutions for the conjugate equations are unique up to unitary equivalence.
Throughout this article, we always fix standard solutions for all $\alpha\in\Irr(\mathcal{C})$, and extend by naturality to arbitrary objects $\alpha\in\mathcal{C}$ (see \cite[Definition~2.2.14]{neshveyev-tuset}).
The positive real number defined by $d(\alpha)=t_\alpha^*t_\alpha=s_\alpha^*s_\alpha$ is referred to as the \textit{quantum dimension} of $\alpha$.

These standard solutions also give rise to canonical tracial functionals $\Tr_\alpha$ on $(\alpha,\alpha)$ via
\[
\Tr_\alpha(T)=s_\alpha^*(T\otimes 1)s_\alpha=t_\alpha^*(1 \otimes T)t_\alpha \; .
\]
Note that these traces are typically not normalized, since $\Tr_\alpha(1)=d(\alpha)$.
It is sometimes convenient to work with the \textit{partial traces} defined by
\begin{align*}
&\Tr_\alpha\otimes\id: (\alpha\beta,\alpha\gamma)\to (\beta,\gamma):
T\mapsto (t_\alpha^*\otimes 1)(1 \otimes T)(t_\alpha\otimes 1) \; ,\\
&\id\otimes\Tr_\alpha: (\beta\alpha,\gamma\alpha)\to (\beta,\gamma):
T\mapsto (1 \otimes s_\alpha^*)(T\otimes 1)(1 \ot s_\alpha) \; .
\end{align*}
for $\alpha,\beta,\gamma\in\mathcal{C}$.
These satisfy $\Tr_\beta\circ (\Tr_\alpha\otimes\id)=\Tr_{\alpha\beta}=\Tr_\alpha\circ(\id\otimes\Tr_\beta)$.
For all $\alpha,\beta\in\mathcal{C}$, the categorical traces induce an inner product on $(\alpha,\beta)$, given by
\begin{equation}
\label{eqn:morphism-space-innprod}
\langle T,S\rangle = \Tr_\alpha(TS^*) = \Tr_\beta(S^*T).
\end{equation}
Throughout, the notation $\onb(\alpha,\beta)$ will refer to some choice of orthonormal basis of $(\alpha,\beta)$ with respect to this inner product.
Finally, the standard solutions of the conjugate equations induce the \emph{Frobenius reciprocity maps}, which are the unitary isomorphisms given by
\begin{equation}\label{eq.frobenius-reciprocity-maps}
\begin{split}
&(\alpha\beta,\gamma) \to (\alpha,\gamma\overline{\beta}): T\mapsto (1 \otimes s_\beta^*)(T\otimes 1) \; ,\\
&(\alpha\beta,\gamma)\to (\beta,\overline{\alpha}\gamma): T\mapsto (t_\alpha^*\otimes 1)(1 \otimes T) \; ,
\end{split}
\end{equation}
where $\alpha,\beta,\gamma\in\Irr(\mathcal{C})$.

The \textit{tube algebra} $\mathcal{A}$ of a rigid $C^*$-tensor category was first defined by Ocneanu in \cite{ocneanu-chirality} for categories with finitely many irreducibles.
For convenience, we recall some of the exposition from \cite{psv-cohom}.
The tube algebra is defined by the vector space direct sum
\[
\mathcal{A}=\bigoplus_{i,j,\alpha\in\Irr(\mathcal{C})} (i\alpha,\alpha j) \; .
\]
For general $\alpha\in\mathcal{C}$ and $i,j\in\Irr(\mathcal{C})$, a morphism $V\in (i\alpha,\alpha j)$ also defines an element of $\mathcal{A}$ via
\begin{equation}
\label{eqn:general-intertw-in-tube-algebra}
V\mapsto\sum_{\gamma\in\Irr(\mathcal{C})}d(\gamma)\sum_{W\in\onb(\alpha,\gamma)} (1 \otimes W^*)V(W\otimes 1).
\end{equation}
It should be noted that this map is generally not an embedding of $(i\alpha,\alpha j)$ into $\mathcal{A}$.
One easily checks that $\mathcal{A}$ is a $*$-algebra for the following operations
\begin{align*}
V\cdot W&=\delta_{j,j'}(V\otimes 1)(1 \otimes W)\in (i\alpha\beta,\alpha\beta k) \; ,\\
V^\# &= (t_\alpha^*\otimes 1 \otimes 1)(1 \otimes V^*\otimes 1)(1 \otimes 1 \otimes s_\alpha)\in (j\overline{\alpha},\overline{\alpha} i) \; ,
\end{align*}
where $V\in(i\alpha,\alpha j)$, $W\in (j'\beta, \beta k)$ and where the map in \eqref{eqn:general-intertw-in-tube-algebra} is used to view $V \cdot W$ as an element of $\cA$.
We follow the notational convention from \cite{psv-cohom} and explicitly denote the tube algebra operations by $\cdot$ and ${\ }^\#$, to avoid confusion with composition and adjunction of morphisms. It should be noted that $\mathcal{A}$ is not unital, unless $\Irr(\mathcal{C})$ is finite.

For $i\in\Irr(\mathcal{C})$, the identity map on $i$ is an element of $(i\eps,\eps i)$. So it can be considered as an element $p_i\in\mathcal{A}$.
As the notation suggests, $p_i$ is a self-adjoint idempotent in $\mathcal{A}$, and it is easy to see that $p_i\cdot V\cdot p_j=\delta_{ik}\delta_{jk'} V$ when $V\in (k\alpha,\alpha k')$.
The corner $p_i\cdot\mathcal{A}\cdot p_i$ is a unital $*$-algebra and the projections $p_i$, $i \in \Irr(\cC)$, serve as local units for $\cA$. In particular, for all purposes of homological algebra, we can work with $\cA$ as if it were a unital algebra.

The corner $p_\eps\cdot\mathcal{A}\cdot p_\eps$ is canonically isomorphic to the \textit{fusion $*$-algebra}~$\CC[\mathcal{C}]$.
This algebra is formed by taking the free vector space over $\Irr(\mathcal{C})$, and defining multiplication by the fusion rules, i.e.
\[
\alpha\cdot \beta = \sum_{\gamma\in\Irr(\mathcal{C})} \mult(\gamma,\alpha\otimes\beta)\gamma \; .
\]
The involution on $\CC[\mathcal{C}]$ is given by conjugation in $\mathcal{C}$.

The tube algebra comes with a faithful trace $\tau$ (see \cite[Proposition~3.10]{psv-cohom}).
For $V\in (i\alpha,\alpha j)$ with $i,j,\alpha\in\Irr(\mathcal{C})$, this trace is given by
\[
\tau(V)=\begin{cases}
\Tr_i(V) & i=j,\alpha=\eps\\
0 & \text{otherwise} \; .
\end{cases}
\]
In \cite{psv-cohom}, it is also shown that every involutive action of $\mathcal{A}$ on a pre-Hilbert space is automatically by bounded operators.
In particular, this allows us to define a von Neumann algebra $\mathcal{A}''$ by considering the faithful action of $\mathcal{A}$ on $L^2(\mathcal{A},\tau)$ by left multiplication, and then taking the bicommutant.
Additionally, the trace $\tau$ uniquely extends to a faithful normal semifinite trace on $\mathcal{A}''$.

For $i,j,\alpha\in\Irr(\mathcal{C})$, we now have two inner products on $(i\alpha,\alpha j)$, related by
\[
\Tr_{\alpha j}(W^* V) = d(\alpha) \, \tau(W^\#\cdot V)\; .
\]
We will however always work with the inner product given by $\Tr_{\al j}(W^* V)$, because it is compatible with the inner product in \eqref{eqn:morphism-space-innprod}, which is defined on all spaces of intertwiners and which makes the Frobenius reciprocity maps \eqref{eq.frobenius-reciprocity-maps} unitary.

\subsection{\boldmath Representation theory for rigid $C^*$-tensor categories}
\label{sec:representation-theory}

The unitary representation theory for rigid $C^*$-tensor categories was introduced in \cite{pv-repr-subfactors} and several equivalent formulations were found in \cite{neshveyev-yamashita,ghosh-jones,psv-cohom}. Following \cite{ghosh-jones}, a unitary representation of $\cC$ is given by a \emph{nondegenerate $*$-representation of the tube algebra} of $\mathcal{C}$. Following \cite{neshveyev-yamashita}, a unitary representation of $\cC$ is given by a \textit{unitary half braiding on an ind-object of} $\mathcal{C}$, i.e.\ an object in the \textit{unitary Drinfeld center} $\mathcal{Z}(\ind{\mathcal{C}})$. Here, the category $\ind{\mathcal{C}}$ may be thought of as a completion of $\mathcal{C}$ with infinite direct sums, giving rise to a (nonrigid) $C^*$-tensor category.
A \textit{unitary half braiding} on an ind-object $X\in\ind{\mathcal{C}}$ is a natural unitary isomorphism $\sigma_{-}:-\otimes X\to X\otimes -$ that satisfies the half braiding condition
\[
\sigma_{Y\otimes Z}=(\sigma_Y\otimes 1)(1 \otimes\sigma_Z)
\]
for all $Y,Z\in\ind{\mathcal{C}}$.
The collection of unitary half braidings on $\ind{\mathcal{C}}$ is denoted by $\mathcal{Z}(\ind{\mathcal{C}})$. We refer to \cite{neshveyev-yamashita} for rigorous definitions and basic properties of these objects.

By \cite[Proposition~3.14]{psv-cohom}, there is the following bijective correspondence between nondegenerate right Hilbert $\cA$-modules $\cK$ and unitary half braidings $(X,\sigma)$.
Given $(X,\sigma)\in\mathcal{Z}(\ind{\mathcal{C}})$, one defines $\mathcal{K}$ as the Hilbert space direct sum of the Hilbert spaces $(X,i)$, $i \in \Irr(\cC)$.
To turn $\mathcal{K}$ into a right $\mathcal{A}$-module, we let $V\in (i\alpha,\alpha j)$ act on a vector $\xi\in (X,i')$ by
\begin{equation}
\label{eqn:ind-obj-a-module-action}
\xi\cdot V = \delta_{ii'}(\Tr_\alpha\otimes\id)(\sigma_\alpha^*(\xi\otimes 1) V) \in (X,j) \; .
\end{equation}
In particular, we see that $\cK \cdot p_i = (X,i)$.

\subsection{\boldmath (Co)homology and $L^2$-Betti numbers for rigid $C^*$-tensor categories}

(Co)homology for rigid $C^*$-tensor categories was introduced in \cite{psv-cohom}. One of the equivalent ways to describe this (co)homology theory is as Hochschild (co)homology for the tube algebra $\cA$, see \cite[\S~7.2]{psv-cohom}. Concretely, we equip $\mathcal{A}$ with the augmentation (or counit)
\[
\counit:\mathcal{A}\to\CC: V\in (i\alpha,\alpha j)\mapsto \delta_{ij\eps} \Tr_\alpha(V) \; .
\]
Since $\counit$ is a $*$-homomorphism, we can view $\C$ as an $\cA$-module, which should be considered as the trivial representation of $\mathcal{C}$. Let $\mathcal{K}$ be a nondegenerate right Hilbert $\mathcal{A}$-module. We denote the (algebraic) linear span of $\mathcal{K}\cdot p_i$ for $i\in\Irr(\mathcal{C})$ by $\mathcal{K}^0$. Following \cite{psv-cohom}, the homology of $\mathcal{C}$ with coefficients in $\mathcal{K}^0$ is then defined by
\[
H_\bullet(\mathcal{C},\mathcal{K}^0)=\Tor^{\mathcal{A}}_\bullet(\mathcal{K}^0,\CC) \; .
\]
Similarly, the cohomology of $\mathcal{C}$ with coefficients in $\mathcal{K}^0$ is given by
\[
H^\bullet(\mathcal{C},\mathcal{K}^0)=\Ext_{\mathcal{A}}^\bullet(\CC,\mathcal{K}^0) \; .
\]
Note that, in the special case where $\mathcal{K}=L^2(\mathcal{A})$, the left $\mathcal{A}''$-module structure on $L^2(\mathcal{A})$ induces a natural left $\mathcal{A}''$-module structure on the (co)homology spaces. As in \cite{psv-cohom}, one then defines the $n$-th $L^2$-\textit{Betti number} of $\mathcal{C}$ as
\begin{equation}
\label{eqn:betti-number-definition}
\beta^{(2)}_n(\mathcal{C})=\dim_{\mathcal{A}''} H^n(\mathcal{C},L^2(\mathcal{A})^0)=\dim_{\mathcal{A}''} H_n(\mathcal{C},L^2(\mathcal{A})^0).
\end{equation}
where $\dim_{\mathcal{A}''}$ is the L\"{u}ck dimension with respect to the normal semifinite trace $\tau$ on $\mathcal{A}''$.

We refer to \cite[\S~6.1]{lueck-l2-invariants}, \cite[\S~A.4]{kpv13} and Remark \ref{rem.dimension-function} for the relevant definitions and properties of the dimension function $\dim_N$ on arbitrary $N$-modules, associated with a von Neumann algebra $N$ equipped with a faithful normal semifinite trace $\Tr$. Note that the second equality in \eqref{eqn:betti-number-definition} is nontrivial and was proved in \cite[Proposition~6.4]{psv-cohom}. When $\mathcal{C}$ is a discrete group, all these notions reduce to the familiar ones for groups.

\section{\boldmath A scaling formula for $L^2$-Betti numbers}

\subsection{Index of a subcategory}

\begin{definition}\label{def:general-index}
Let $\mathcal{C}$ be a rigid $C^*$-tensor category, and $\mathcal{C}_1\subset\mathcal{C}$ a full $C^*$-tensor subcategory of $\mathcal{C}$.
For an object $\alpha\in\mathcal{C}$, we define $[\alpha]_{\mathcal{C}_1}$ as the largest subobject of $\alpha$ that belongs to $\mathcal{C}_1$.
We denote the orthogonal projection of $\alpha$ onto $[\alpha]_{\mathcal{C}_1}$ by $P^\alpha_{\mathcal{C}_1} \in (\al,\al)$.
Fixing $\alpha\in\Irr(\mathcal{C})$, we define the $\mathcal{C}_1$-\textit{orbit} of $\alpha$ as
\[
\alpha\cdot\mathcal{C}_1=\{\beta\in\Irr(\mathcal{C})\mid \exists\gamma\in\mathcal{C}_1: \beta\prec\alpha\gamma\} \; .
\]
Note that in this definition, we can replace $\mathcal{C}_1$ by $\Irr(\mathcal{C}_1)$ without changing the orbit.
By Frobenius reciprocity, the orbits form a partition of $\Irr(\mathcal{C})$.
If $\alpha_1,\ldots,\alpha_k$ are representatives of $\mathcal{C}_1$-orbits, the \textit{index} of $\mathcal{C}_1\subset\mathcal{C}$ is defined as
\begin{equation}
\label{eqn:general-index}
{}[\mathcal{C}:\mathcal{C}_1]=\sum_{i=1}^k\frac{d(\alpha_i)^2}{d([\overline{\alpha_i}\alpha_i]_{\mathcal{C}_1})}.
\end{equation}
If the set of orbits is infinite, we put $[\mathcal{C}:\mathcal{C}_1]=\infty$.
\end{definition}

In Lemma \ref{lem:index-well-defined}, we show that the index is well defined. In Proposition \ref{prop:tube-algebra-projective-markov}, we prove that $[\cC : \cC_1]$ equals the Jones index for an associated inclusion of von Neumann algebra completions of tube algebras. In Proposition \ref{prop.index-transitive}, we prove the formula $[\cC : \cC_2] = [\cC : \cC_1] \, [\cC_1 : \cC_2]$ when $\cC_2 \subset \cC_1 \subset \cC$. So, the above definition of $[\cC : \cC_1]$ is indeed natural.

When $\mathcal{C}_1=\{\eps\}$, the index defined above coincides with the global index $d(\cC)$ of $\mathcal{C}$. When $\cC$ has only finitely many irreducible objects, we have $[\cC : \cC_1] = d(\cC) / d(\cC_1)$, see Proposition \ref{prop.index-transitive}.

Another extreme situation arises when
\begin{equation}
\label{eqn:unit-radical-definition}
N(\mathcal{C})=\{\gamma\in\Irr(\mathcal{C})\mid \exists \alpha_1,\ldots,\alpha_k\in\Irr(\mathcal{C}): \gamma\prec \alpha_1\cdots\alpha_k\overline{\alpha_k}\cdots\overline{\alpha_1}\}
\end{equation}
is a subset of $\Irr(\mathcal{C}_1)$. In this case, the index simply counts the number of orbits. In particular, we recover the index for subgroups when $\mathcal{C}_1\subset\mathcal{C}$ are both groups considered as $C^*$-tensor categories.

\begin{lemma} \label{lem:index-well-defined}
Let $\mathcal{C}$ be a rigid $C^*$-tensor category with full $C^*$-tensor subcategory $\mathcal{C}_1$.
Then, for $\alpha,\beta\in\Irr(\mathcal{C})$ with $\beta\in\alpha\cdot\mathcal{C}_1$, we have that
\begin{equation}
\label{eqn:index-term-orbit-invariant}
\frac{d([\overline{\alpha}\alpha]_{\mathcal{C}_1})}{d(\alpha)^2}
=\frac{d([\overline{\alpha}\beta]_{\mathcal{C}_1})}{d(\alpha)d(\beta)}
=\frac{d([\overline{\beta}\beta]_{\mathcal{C}_1})}{d(\beta)^2}.
\end{equation}
\end{lemma}

\begin{proof}
For arbitrary $\alpha,\beta\in\Irr(\mathcal{C})$, we have that
\begin{align*}
(\Tr_{\overline{\alpha}}\otimes\id)(P^{\overline{\alpha}\beta}_{\mathcal{C}_1})=d(\beta)^{-1}\Tr_{\overline{\alpha}\beta}(P^{\overline{\alpha}\beta}_{\mathcal{C}_1}) \, 1 =\frac{d([\overline{\alpha}\beta]_{\mathcal{C}_1})}{d(\beta)} \, 1 \; ,
\end{align*}
by irreducibility of $\beta$.
Now suppose that $\alpha,\beta$ satisfy the conditions of the lemma.
Choose $\gamma\in\Irr(\mathcal{C}_1)$ such that $\beta\prec\alpha\gamma$.
For any isometry $W:\beta\to\alpha\gamma$, we compute
\begin{align*}
W(\Tr_{\overline{\alpha}}\otimes\id)(P^{\overline{\alpha}\beta}_{\mathcal{C}_1})
&=(\Tr_{\overline{\alpha}}\otimes\id)((1 \otimes W)P^{\overline{\alpha}\beta}_{\mathcal{C}_1})\\
&=(\Tr_{\overline{\alpha}}\otimes\id)(P^{\overline{\alpha}\alpha\gamma}_{\mathcal{C}_1}(1 \otimes W))\\
&=((\Tr_{\overline{\alpha}}\otimes\id)(P^{\overline{\alpha}\alpha}_{\mathcal{C}_1})\otimes 1)W\\
&=\frac{d([\overline{\alpha}\alpha]_{\mathcal{C}_1})}{d(\alpha)}W \; ,
\end{align*}
where we used that $P^{\overline{\alpha}\alpha\gamma}_{\mathcal{C}_1}=P^{\overline{\alpha}\alpha}_{\mathcal{C}_1}\otimes 1$, as is easy to see by splitting $\overline{\alpha}\alpha$ into irreducible components.
Multiplying by $W^*$ on the left, we find that
\[
(\Tr_{\overline{\alpha}}\otimes\id)(P^{\overline{\alpha}\beta}_{\mathcal{C}_1})=\frac{d([\overline{\alpha}\alpha]_{\mathcal{C}_1})}{d(\alpha)} 1  \;  .
\]
We already proved that the left-hand side equals $\frac{d([\overline{\alpha}\beta]_{\mathcal{C}_1})}{d(\beta)} \, 1$. So, the first equality in \eqref{eqn:index-term-orbit-invariant} follows. The second one is proven analogously.
\end{proof}

\begin{proposition}\label{prop.index-transitive}
Let $\cC$ be a rigid $C^*$-tensor category with full $C^*$-tensor subcategories $\cC_2 \subset \cC_1 \subset \cC$. Then,
$$[\cC : \cC_2] = [\cC : \cC_1] \, [\cC_1 : \cC_2] \; .$$
In particular, if $\cC$ is a rigid $C^*$-tensor category with finitely many irreducible objects and if $\cC_1 \subset \cC$ is a full $C^*$-tensor subcategory, then $[\cC : \cC_1] = d(\cC)/d(\cC_1)$, where $d(\cC)$ and $d(\cC_1)$ denote the global index of $\cC$ and $\cC_1$.
\end{proposition}

Since a short proof for Proposition \ref{prop.index-transitive} can be given using the language of Markov inclusions, we postpone the proof until the end of Section \ref{sec.scaling}.

In the concrete computations of $L^2$-Betti numbers in this paper, we only need the particularly easy tensor subcategories $\mathcal{C}_1 \subset \mathcal{C}$ that arise from a homomorphism to a finite group. More precisely, assume that we are given a group $\Lambda$ and a map $\Xi:\Irr(\mathcal{C})\to\Lambda$ satisfying the following two properties.
\begin{enumlist}
\item For all $\alpha,\beta,\gamma\in\Irr(\mathcal{C})$ with $\gamma \prec \alpha\beta$, we have $\Xi(\gamma)=\Xi(\alpha)\Xi(\beta)$.
\item For all $\alpha\in\Irr(\mathcal{C})$, we have $\Xi(\overline{\alpha})=\Xi(\alpha)^{-1}$.
\end{enumlist}
Defining $\CatKer(\Xi) \subset \cC$ as those objects in $\cC$ that can be written as a direct sum of irreducible objects $\gamma \in \Irr(\cC)$ with $\Xi(\gamma) = e$, we obtain a full $C^*$-tensor subcategory $\CatKer(\Xi) \subset \cC$ of index $|\Lambda|$.

Note that $N(\mathcal{C})$, as defined in \eqref{eqn:unit-radical-definition}, always is a subset of $\CatKer(\Xi)$. Actually, denoting by $\Gamma$ the set of orbits for the left (or right) action of $N(\cC)$ on $\Irr(\cC)$, we get that $\Gamma$ has a natural group structure and we can view $\Gamma$ as the largest group quotient of $\cC$.

\subsection{Markov inclusions of tracial von Neumann algebras}\label{sec.Markov}

In \cite[Section 1.1.4]{popa-classification}, the concept of a $\lambda$-Markov inclusion $N \subset (M,\tau)$ of tracial von Neumann algebras was introduced. More generally, Popa defined in \cite[Section 1.2]{popa-cbms} the $\lambda$-Markov property for arbitrary inclusions of von Neumann algebras $N \subset M$ together with a faithful normal conditional expectation $E : M \recht N$. Taking in the tracial setting the unique trace-preserving conditional expectation, both notions coincide.

In this paper, we need a slight variant of this concept for inclusions $N \subset M$ where both $N$ and $M$ are equipped with fixed faithful normal semifinite traces, denoted $\Tr_N$ and $\Tr_M$, but the inclusion need not be trace-preserving. In particular, there is no canonical conditional expectation of $M$ onto $N$.

Recall that an element $v \in M$ is called right $N$-bounded if there exists a $\kappa > 0$ such that $\Tr_M(a^* v^* v a) \leq \kappa \, \Tr_N(a^* a)$ for all $a \in N$. We denote by $L_v : L^2(N,\Tr_N) \recht L^2(M,\Tr_M)$ the associated bounded operator, which is right $N$-linear and given by $L_v(a) = va$ for all $a \in N \cap L^2(N,\Tr_N)$. A family $(v_i)_{i \in I}$ of right $N$-bounded vectors in $M$ is called a Pimsner-Popa basis for $N \subset M$ if
$$\sum_{i \in I} L_{v_i} L_{v_i}^* = 1 \; .$$

\begin{definition}\label{def.Markov-inclusion}
Let $(N,\Tr_N)$ and $(M,\Tr_M)$ be von Neumann algebras equipped with faithful normal semifinite traces. Assume that $N \subset M$, but without assuming that this inclusion is trace-preserving. We say that the inclusion is $\lambda$-Markov for a given number $\lambda > 0$ if a Pimsner-Popa basis $(v_i)_{i \in I}$ satisfies
$$\sum_{i \in I} v_i v_i^* = \lambda^{-1} \, 1 \; .$$
\end{definition}

One checks that this definition does not depend on the choice of the Pimsner-Popa basis.

\begin{definition}\label{def.locally-finite}
Given a von Neumann algebra $M$ equipped with a faithful normal semifinite trace $\Tr$, we call an (algebraic) right $M$-module $\cE$ \emph{locally finite} if for every $\xi \in \cE$, there exists a projection $p \in M$ with $\Tr(p) < \infty$ and $\xi = \xi p$.
\end{definition}

Note that for every projection $p \in M$ with $\Tr(p) < \infty$, the right $M$-module $p L^2(M)$ is locally finite, because for every $\xi \in L^2(M)$, the right support projection of $p \xi$ has finite trace.

For our computations, the following scaling formula is essential.

\begin{proposition}\label{prop.dim-scaling-Markov-inclusion}
Let $(N,\Tr_N)$ and $(M,\Tr_M)$ be von Neumann algebras equipped with faithful normal semifinite traces. Assume that $N \subset M$ and that $\lambda > 0$. The inclusion is $\lambda$-Markov if and only if $\dim_M(\cE) = \lambda \, \dim_N(\cE)$ for every locally finite $M$-module $\cE$.

We have $\dim_M(\cE) = \lambda \, \dim_N(\cE)$ for arbitrary $M$-modules $\cE$ if and only if the inclusion is $\lambda$-Markov and the restriction of $\Tr_M$ to $N$ is semifinite.
\end{proposition}

\begin{proof}
Fix a Pimsner-Popa basis $(v_i)_{i \in I}$ for $N \subset M$, w.r.t.\ the traces $\Tr_N$, $\Tr_M$. Define the projection $q \in B(\ell^2(I)) \ovt N$ given by $q_{ij} = L_{v_i}^* L_{v_j}$. Then,
$$U : L^2(M,\Tr_M) \recht q(\ell^2(I) \ot L^2(N,\Tr_N)) : U(x) = \sum_{i \in I} e_i \ot L_{v_i}^*(x)$$
is a well-defined right $N$-linear unitary operator. Whenever $a \in M$, the operator $U a U^*$ commutes with the right $N$-action and so, we get a well-defined unital $*$-homomorphism
$$\al : M \recht q (B(\ell^2(I)) \ovt N) q : \al(a) = U a U^* \; .$$
A direct computation gives that
$$(\Tr \ot \Tr_N)(\al(a)) = \sum_{i \in I} \Tr_M(v_i^* a v_i)$$
for all $a \in M^+$. So the inclusion $N \subset M$ is $\lambda$-Markov if and only if $\Tr_M(p) = \lambda \, (\Tr \ot \Tr_N)(\al(p))$ for every projection $p \in M$. Note that the left-hand side equals $\dim_{-M}(p L^2(M))$, while the right-hand side equals $\lambda \dim_{-N}(p L^2(M))$. So if the formula $\dim_M(\cE) = \lambda \, \dim_N(\cE)$ holds for all locally finite $M$-modules, it holds in particular for $\cE = p L^2(M)$ for every projection $p \in M$ with $\Tr_M(p) < \infty$ and we conclude that $\Tr_M(p) = \lambda \, (\Tr \ot \Tr_N)(\al(p))$ for every projection $p \in M$ with $\Tr_M(p) < \infty$. An arbitrary projection $p \in M$ can be written as the limit of an increasing net of finite trace projections, so that the same formula holds for all projections $p \in M$ and thus, $N \subset M$ is $\lambda$-Markov.

Conversely, assume that $N \subset M$ is $\lambda$-Markov. We prove that $\dim_M(\cE) = \lambda \, \dim_N(\cE)$ for every locally finite $M$-module $\cE$. Denote by $\cL_1$ the class of $M$-modules that are isomorphic with $p (\C^n \ot M)$ for some $n \in \N$ and some projection $p \in M_n(\C) \ot M$ having finite trace. We start by proving that $\dim_N(\cE) = \lambda^{-1} \dim_M(\cE)$ for all $\cE \in \cL_1$.

Take a finite trace projection $p \in M_n(\C) \ot M$ such that $\cE \cong p(\C^n \ot M)$. We have
$$p(\C^n \ot M) \subset p(\C^n \ot L^2(M,\Tr_M)) \cong (\id \ot \al)(p) (\C^n \ot \ell^2(I) \ot L^2(N,\Tr_N)) \; .$$
Therefore,
\begin{align*}
\dim_{-N}(p (\C^n \ot M)) & \leq \dim_{-N}((\id \ot \al)(p) (\C^n \ot \ell^2(I) \ot L^2(N,\Tr_N))) \\ &= (\Tr \ot \Tr \ot \Tr_N)(\id \ot \al)(p) \\
& = \lambda^{-1} \, (\Tr \ot \Tr_M)(p) = \lambda^{-1} \, \dim_{-M}(p (\C^n \ot M)) \; .
\end{align*}
Conversely, since
$$(1 \ot U^*)\bigl( (\id \ot \al)(p) (e_i \ot e_j \ot a) \bigr) = p(e_i \ot v_j a)$$
for all $i \in \{1,\ldots,n\}$, $j \in I$ and $a \in N$, we get for every finite subset $I_0 \subset I$ the injective $N$-module map
$$(\id \ot \al)(p) (\C^n \ot \ell^2(I_0) \ot N) \hookrightarrow p(\C^n \ot M) \; .$$
Letting $I_0$ increase and taking $\dim_{-N}$, it follows that
$$(\Tr \ot \Tr \ot \Tr_N)(\id \ot \al)(p) \leq \dim_{-N}(p(\C^n \ot M)) \; .$$
The left-hand side equals $\lambda^{-1} \, (\Tr \ot \Tr_M)(p) = \lambda^{-1} \, \dim_{-M}(p(\C^n \ot M))$. In combination with the converse inequality above, we have proved that $\dim_M(\cE) = \lambda \, \dim_N(\cE)$ for every $\cE \in \cL_1$.

Next denote by $\cL_2$ the class of all $M$-modules that arise as the quotient of an $M$-module in $\cL_1$. Let $\cE \in \cL_2$ and let $0 \recht \cE_0 \recht \cE_1 \recht \cE \recht 0$ be an exact sequence of $M$-modules, with $\cE_1 \in \cL_1$. Since every finitely generated $M$-submodule of an $M$-module in $\cL_1$ again belongs to $\cL_1$, we can write $\cE_0$ as the union of an increasing family of $M$-submodules $\cE_j \subset \cE_0$ with $\cE_j \in \cL_1$ for all $j$. Since both $\dim_M$ and $\dim_N$ are continuous when taking increasing unions (see Remark \ref{rem.dim-function-continuity-additivity}), we get that $\dim_M(\cE_0) = \lambda \, \dim_N(\cE_0)$. Since both $\dim_M$ and $\dim_N$ are additive with respect to short exact sequences (see Remark \ref{rem.dim-function-continuity-additivity} as well), we conclude that
$$\dim_M(\cE) = \dim_M(\cE_1) - \dim_M(\cE_0) = \lambda \, \dim_N(\cE_1) - \lambda \, \dim_N(\cE_0) = \lambda \, \dim_N(\cE) \; .$$
Finally, every locally finite $M$-module can be written as the union of an increasing family of $M$-submodules in $\cL_2$. So again using the continuity of the dimension function, we find that $\dim_M(\cE) = \lambda \, \dim_N(\cE)$ for all locally finite $M$-modules $\cE$.

Next assume that $N \subset M$ is $\lambda$-Markov and that the restriction of $\Tr_M$ to $N$ is semifinite. We can then choose an increasing net of projections $p_n \in N$, converging to $1$ strongly, with $\Tr_N(p_n) < \infty$ and $\Tr_M(p_n) < \infty$ for all $n$. Let $\cE$ be an arbitrary $M$-module. By \cite[Lemmas A.15 and A.16]{kpv13}, we have $\dim_M(\cE) = \lim_n \dim_M(\cE \, p_n M)$. For each $n$, the $M$-module $\cE \, p_n M$ is locally finite. Therefore, $\dim_M(\cE \, p_n M) = \lambda \, \dim_N(\cE \, p_n M)$. Since $\cE \, p_n M \subset \cE$, it follows that $\dim_M(\cE) \leq \lambda \, \dim_N(\cE)$. Conversely, $\cE \, p_n N \subset \cE \, p_n M$, so that
$$\dim_N(\cE \, p_n M) \geq \dim_N(\cE \, p_n N) \; .$$
Again using \cite[Lemmas A.15 and A.16]{kpv13}, we have $\lim_n \dim_N(\cE \, p_n N) = \dim_N(\cE)$, so that the inequality $\dim_M(\cE) \geq \lambda \, \dim_N(\cE)$ follows.

Finally, assume that the restriction of $\Tr_M$ to $N$ is not semifinite. We then find a nonzero projection $p \in N$ such that $\Tr_M(x) = +\infty$ for every nonzero element $x \in p N^+ p$. Define the two-sided ideal $M_0 \subset M$ consisting of all elements $x \in M$ whose left (equivalently right) support projection has finite $\Tr_M$. Define $\cE = M/M_0$ and view $\cE$ as a right $M$-module. Whenever $p \in M$ is a projection with $\Tr_M(p) < \infty$, we have that $\cE \, p = \{0\}$. By \cite[Definition A.14]{kpv13}, we have that $\dim_{-M}(\cE) = 0$. On the other hand, the map $p N \recht \cE : x \mapsto x + M_0$ is $N$-linear and injective because $p N \cap M_0 = \{0\}$. Therefore, $\dim_N(\cE) \geq \Tr_N(p) > 0$. So, the dimension scaling formula fails in general when the restriction of $\Tr_M$ to $N$ is no longer semifinite.
\end{proof}

\begin{remark}\label{rem.dim-function-continuity-additivity}
In the proof of Proposition \ref{prop.dim-scaling-Markov-inclusion}, we made use of the following continuity and additivity property of the dimension function $\dim_M$ associated with a von Neumann algebra $M$ equipped with a faithful normal semifinite trace $\Tr$.
\begin{enumlist}
\item\label{dim.one} Assume that $\cE$ is an $M$-module and $\cE_j \subset \cE$ is an increasing net of $M$-submodules with $\bigcup_j \cE_j = \cE$. Then, $\dim_M(\cE) = \lim_j \dim_M(\cE_j)$.
\item\label{dim.two} Assume that $0 \recht \cE_1 \recht \cE \recht \cE_2 \recht 0$ is an exact sequence of $M$-modules. Then, $\dim_M(\cE) = \dim_M(\cE_1) + \dim_M(\cE_2)$.
\end{enumlist}
When $\Tr$ is a tracial state, meaning that $\Tr(1) = 1$, these properties are proved in \cite[Theorem 6.7(4)]{lueck-l2-invariants}. When $\Tr$ is semifinite, we can take an increasing net of projections $p_i \in M$ with $\Tr(p_i) < \infty$ for all $i$ and $p_i \recht 1$ strongly. Define the tracial state $\tau_i$ on $p_i M p_i$ given by $\tau_i(x) = \Tr(p_i)^{-1} \Tr(x)$. Then \cite[Lemma A.16]{kpv13} says that for every $M$-module $\cE$, the net $\Tr(p_i) \dim_{(p_i M p_i,\tau_i)}(\cE p_i)$ is increasing and converges to $\dim_M(\cE)$. Therefore, the continuity and additivity properties \ref{dim.one} and \ref{dim.two} above are also valid for $\dim_M$.
\end{remark}

\begin{remark}\label{rem.dimension-function}
Let $(M,\Tr)$ be a von Neumann algebra equipped with a faithful normal semifinite trace. Proposition \ref{prop.dim-scaling-Markov-inclusion} shows that the dimension function $\dim_M$ has a subtle behavior. We therefore also want to clarify why \cite[Definition A.14]{kpv13}, given by
\begin{equation}\label{eq.kpv-def}
\dim_M(\cE) = \sup \bigl\{ \Tr(q) \, \dim_{qMq}(\cE \, q) \bigm| q \in M \;\;\text{is a projection with}\;\; \Tr(q) <\infty \bigr\}
\end{equation}
and making use of the dimension function for $(qMq, \Tr(q)^{-1} \Tr(\,\cdot\,))$, coincides with \cite[Definition B.17]{petersen-PhD}, given by
\begin{equation}\label{eq.pet-def}
\begin{split}
\dim_M(\cE) = \sup \bigl\{ (\Tr \ot \Tr)(p) \bigm| \; & p \in M_n(\C) \ot M \;\;\text{is a projection with finite trace and}\\ & p (\C^n \ot M) \hookrightarrow \cE \;\;\text{as $M$-modules} \bigr\}
\end{split}
\end{equation}
Whenever $p(\C^n \ot M) \hookrightarrow \cE$, we have $p(\C^n \ot M q) \hookrightarrow \cE \, q$. Denoting by $z_q \in \cZ(M)$ the central support of $q$, it follows from \cite[Lemma A.15]{kpv13} that
$$\Tr(q) \, \dim_{qMq}(\cE \, q) \geq \Tr(q) \, \dim_{qMq}(p(\C^n \ot Mq)) = (\Tr \ot \Tr)(p(1 \ot z_q)) \; .$$
Taking the supremum over all finite trace projections $q \in M$ and all embeddings $p(\C^n \ot M) \hookrightarrow \cE$, it follows that the dimension in \eqref{eq.pet-def} is bounded above by the dimension in \eqref{eq.kpv-def}.

Conversely, $\Tr(q) \, \dim_{qMq}(\cE \, q)$ can be computed as the supremum of $(\Tr \ot \Tr)(p)$ where $p \in M_n(\C) \ot q M q$ is a projection and $\theta : p (\C^n \ot Mq) \hookrightarrow \cE \, q$. Defining $\xi \in \overline{\C^n} \ot \cE \, q$ given by $\xi = \sum_{i=1}^n e_i^* \ot \theta(p(e_i \ot q))$, it follows that $\xi = \xi p$ and $\theta(x) = \xi x$ for all $x \in p(\C^n \ot Mq)$. Then $\psi : p(\C^n \ot M) \recht \cE : \psi(x) = \xi x$ is $M$-linear. We claim that $\psi$ remains injective. Indeed, if $\psi(x) = 0$, then for all $i \in \{1,\ldots,n\}$, also $\theta(xx^*(e_i \ot 1)) = \psi(x) x^* (e_i \ot 1) = 0$. So, $xx^*(e_i \ot 1) = 0$ for all $i$ and thus, $x = 0$. It follows that the dimension in \eqref{eq.kpv-def} is bounded above by the dimension in \eqref{eq.pet-def}.
\end{remark}

\subsection{The scaling formula}\label{sec.scaling}

The goal of this section is to prove the following scaling formula for $L^2$-Betti numbers under finite-index inclusions.

\begin{theorem}\label{thm:scaling-formula}
Let $\mathcal{C}_1\subset\mathcal{C}$ be a finite-index inclusion of rigid $C^*$-tensor categories.
Then
\[
\beta^{(2)}_n(\mathcal{C}_1)=[\mathcal{C}:\mathcal{C}_1]\; \beta^{(2)}_n(\mathcal{C})
\]
for all $n\geq 0$.
\end{theorem}

For the rest of this section, fix a rigid $C^*$-tensor category $\cC$ and a full $C^*$-tensor subcategory $\mathcal{C}_1 \subset \mathcal{C}$. The tube algebra $\cA_1$ of $\cC_1$ naturally is a unital $*$-subalgebra of a \emph{corner} of the tube algebra $\cA$ of $\cC$. In dimension computations, this causes a number of issues that can be avoided by considering the $*$-subalgebra $\cAtil_1 \subset \cA$ given by
\begin{equation}
\label{eqn:a1tilde-def}
\widetilde{\mathcal{A}}_1=\bigoplus_{i,j\in\Irr(\mathcal{C})}\bigoplus_{\alpha\in\Irr(\mathcal{C}_1)} (i\alpha,\alpha j).
\end{equation}
We still have a natural trace $\tau$ on $\cAtil_1$ and the inclusion $\cAtil_1 \subset \cA$ is trace-preserving.

As a first lemma, we prove that the homology of $\cC_1$ can be computed as the Hochschild homology of $\cAtil_1$ with the counit augmentation $\counit : \cAtil_1 \recht \C$.

\begin{lemma}\label{lem.extend-to-Atilde-1}
Define the central projection $p_1$ in the multiplier algebra of $\cAtil_1$ given by $p_1 = \sum_{i \in \Irr(\cC_1)} p_i$. Note that $p_1 \cdot \cAtil_1 \cdot p_1 \cong \cA_1$ naturally.

For every nondegenerate right Hilbert $\cAtil_1$-module $\cK$, there are natural isomorphisms
$$H_\bullet(\cC_1,\cK \cdot p_1) \cong \Tor_\bullet^{\cAtil_1}(\cK^0,\C) \quad\text{and}\quad H^\bullet(\cC_1,\cK \cdot p_1) \cong \Ext^\bullet_{\cAtil_1}(\C,\cK^0) \; .$$
We also have that
$$\beta_n^{(2)}(\cC_1) = \dim_{\cAtil_1''} \Tor_n^{\cAtil_1}(L^2(\cAtil_1)^0,\C) \; .$$
\end{lemma}

\begin{proof}
If $i\in\Irr(\mathcal{C})$ and $\alpha,j\in\Irr(\mathcal{C}_1)$, then $(i\alpha,\alpha j)$ can only be nonzero if $i\in\Irr(\mathcal{C}_1)$, by Frobenius reciprocity.
Interchanging the roles of $i$ and $j$, we conclude that $p_1$ is central in the multiplier algebra $M(\cAtil_1)$.
Because $p_\eps\leq p_1$, it follows that
\begin{align*}
\widetilde{\mathcal{A}}_1\otimes_{\mathcal{B}}\cdots\otimes_{\mathcal{B}}\widetilde{\mathcal{A}}_1\cdot p_\eps
&=p_1\cdot\widetilde{\mathcal{A}}_1\cdot p_1\otimes_{\mathcal{B}}\cdots\otimes_{\mathcal{B}}p_1\cdot\widetilde{\mathcal{A}}_1\cdot p_\eps\\
&\cong\mathcal{A}_1\otimes_{\mathcal{B}_1}\cdots\otimes_{\mathcal{B}_1}\mathcal{A}_1\cdot p_\eps \; ,
\end{align*}
and similarly for the right bar resolution.
Since the bar resolutions associated to $\mathcal{A}_1$ and $\widetilde{\mathcal{A}}_1$ are equal, the respective $\Tor$ and $\Ext$ functors must also be the same.
\end{proof}

The following formula, generalizing \cite[Lemma~3.9]{psv-cohom}, is crucial for us since we deduce from it that $\cA$ is a projective $\cAtil_1$-module and also that in the finite-index case, the inclusion $\cAtil_1'' \subset \cA''$ is $\lambda$-Markov in the sense of Definition \ref{def.Markov-inclusion}.

\begin{lemma}\label{lem:psv-lemma39-generalisation}
For $\al \in \Irr(\cC)$, we denote by $e_{\alpha\cdot\mathcal{C}_1}$ the orthogonal projection of $L^2(\cA)$ onto the closed linear span of all $(i \beta, \beta j)$ with $i,j \in \Irr(\cC)$ and $\beta \in \al \cdot \cC_1$.

Then, for all $i\in\Irr(\mathcal{C})$ and $\alpha\in\Irr(\mathcal{C})$, we have that
\begin{equation}
\label{eqn:psv-lemma39-generalisation}
\sum_{j\in\Irr(\mathcal{C})} \sum_{W\in\onb(i\alpha,\alpha j)} d(j) \; W\cdot e_{\mathcal{C}_1} \cdot W^\#=\frac{d([\overline{\alpha}\alpha]_{\mathcal{C}_1})}{d(\alpha)} \; p_i\cdot e_{\alpha\cdot \mathcal{C}_1}
\end{equation}
as operators on $L^2(\mathcal{A})$.
\end{lemma}

\begin{proof}
Both the left- and the right-hand side of \eqref{eqn:psv-lemma39-generalisation} vanish on $(i_1 \beta, \beta k) \subset L^2(\cA)$ if $i_1 \neq i$. So we fix $k,\beta \in \Irr(\cC)$ and $V \in (i\beta,\beta k)$ and prove that both sides of \eqref{eqn:psv-lemma39-generalisation} agree on $V$.

For every $j\in\Irr(\mathcal{C})$ and $W\in (i\alpha,\alpha j)$, we have
$$(e_{\mathcal{C}_1} \cdot W^\#)(V) = \sum_{\gamma \in \Irr(\cC_1)} \sum_{U \in \onb(\overline{\alpha} \beta,\gamma)} d(\gamma) (1 \ot U^*) (W^\# \ot 1)(1 \ot V)(U \ot 1) \; .$$
We claim that $(e_{\mathcal{C}_1} \cdot W^\#)(V)$ is the image in $\cA$ under the map in \eqref{eqn:general-intertw-in-tube-algebra} of the element
\begin{equation}\label{eq.element-element}
(W^\#\otimes 1)(1 \ot V)(P^{\overline{\alpha}\beta}_{\mathcal{C}_1}\otimes 1) \in (j (\albar \be) , (\albar \be) k) \; .
\end{equation}
The claim follows because that image is given by
\begin{multline*}
\sum_{\gamma \in \Irr(\cC)} \sum_{U \in \onb(\overline{\alpha} \beta,\gamma)} d(\gamma) (1 \ot U^*) (W^\#\otimes 1)(1 \ot V)(P^{\overline{\alpha}\beta}_{\mathcal{C}_1}\otimes 1) (U \ot 1) \\
 = \sum_{\gamma \in \Irr(\cC_1)} \sum_{U \in \onb(\overline{\alpha} \beta,\gamma)} d(\gamma) (1 \ot U^*) (W^\#\otimes 1)(1 \ot V) (U \ot 1) \; ,
\end{multline*}
because $P^{\overline{\alpha}\beta}_{\mathcal{C}_1} U$ equals $0$ when $\gamma \not\in \Irr(\cC_1)$ and equals $U$ when $\gamma \in \Irr(\cC_1)$.

It then follows that $(W\cdot e_{\mathcal{C}_1} \cdot W^\#)(V)$ is the image in $\mathcal{A}$ of the element
\begin{equation}
\label{eqn:psv-lemma39-computation-1}
(W\otimes 1 \otimes 1)(1 \otimes W^\#\otimes 1)(1 \otimes 1 \otimes V)(1 \otimes P^{\overline{\alpha}\beta}_{\mathcal{C}_1}\otimes 1) \in (i (\al\albar \be), (\al \albar \be) k) \; .
\end{equation}
By Frobenius reciprocity,
\[
\{W_Z=(s_\alpha^*\otimes 1 \otimes 1)(1 \otimes Z) \mid Z\in \onb(\overline{\alpha}i\alpha,j)\}
\]
is an orthonormal basis of $(i\alpha,\alpha j)$, and any orthonormal basis can be written in this form.

With this notation, we  find that
\begin{align}
\sum_{j\in\Irr(\mathcal{C})}\sum_{Z\in \onb(\overline{\alpha}i\alpha,j)} &d(j) (W_Z\otimes 1)(1 \otimes W_Z^\#)\notag\\
&=\sum_{j\in\Irr(\mathcal{C})}\sum_{Z\in \onb(\overline{\alpha}i\alpha,j)} d(j)(s_\alpha^*\otimes 1^{\otimes 3})(1 \otimes ZZ^*\otimes 1)(1^{\otimes 3}\otimes s_\alpha)\notag\\
&=s_\alpha^*\otimes 1 \otimes s_\alpha = (1 \ot s_\al)(s_\al^* \ot 1) \; .\label{eqn:psv-lemma39-computation-2}
\end{align}
Combining \eqref{eqn:psv-lemma39-computation-1} and \eqref{eqn:psv-lemma39-computation-2}, we thus obtain
\begin{align*}
\sum_{j\in\Irr(\mathcal{C})} & \sum_{W\in\onb(i\alpha,\alpha j)} d(j) (W\cdot e_{\mathcal{C}_1} \cdot W^\#)(V)\\
&=\sum_{\gamma\in\Irr(\mathcal{C})}\sum_{U\in\onb(\alpha\overline{\alpha}\beta,\gamma)} d(\gamma)(1 \otimes U^*)(1 \otimes s_\alpha\otimes 1)V (s_\alpha^*\otimes 1 \ot 1)(1 \otimes P^{\overline{\alpha}\beta}_{\mathcal{C}_1}\otimes 1)(U\otimes 1) \; .
\end{align*}
Choosing the orthonormal basis of $(\al \albar \be, \gamma)$ by first decomposing $\al \albar$, we see that only one of the $U^*(s_\al \ot 1)$ is nonzero and conclude that
\begin{align*}
\sum_{j\in\Irr(\mathcal{C})} \sum_{W\in\onb(i\alpha,\alpha j)} d(j) (W\cdot e_{\mathcal{C}_1} \cdot W^\#)(V)
&=V(s_\alpha^*\otimes 1 \ot 1)(1 \otimes P^{\overline{\alpha}\beta}_{\mathcal{C}_1}\otimes 1)(s_\alpha\otimes 1 \ot 1)\\
&=V((\Tr_{\overline{\alpha}}\otimes\id)(P^{\overline{\alpha}\beta}_{\mathcal{C}_1})\otimes 1) \; .
\end{align*}
Using Lemma \ref{lem:index-well-defined}, we get that
\[
V((\Tr_{\overline{\alpha}}\otimes\id)(P^{\overline{\alpha}\beta}_{\mathcal{C}_1})\otimes 1)=\frac{d([\overline{\alpha}\beta]_{\mathcal{C}_1}])}{d(\beta)}V=\begin{cases}
\frac{d([\overline{\alpha}\alpha]_{\mathcal{C}_1})}{d(\alpha)}V & \;\;\text{if}\;\; \beta\in\alpha\cdot\mathcal{C}_1 \; ,\\
0 & \;\;\text{otherwise.}
\end{cases}
\]
This concludes the proof of the lemma.
\end{proof}

\begin{proposition}\label{prop:tube-algebra-projective-markov}
Let $\mathcal{C}_1\subset\mathcal{C}$ be a finite-index inclusion of rigid $C^*$-tensor categories. Denote by $\cA$ the tube algebra of $\cC$ and define its subalgebra $\cAtil_1$ as in \eqref{eqn:a1tilde-def}. Then $\mathcal{A}$ is projective as a left $\widetilde{\mathcal{A}}_1$-module and as a right $\cAtil_1$-module. Moreover, the associated inclusion of von Neumann algebras $\widetilde{\mathcal{A}}_1''\subset\mathcal{A}''$ is $\lambda$-Markov with $\lambda=[\mathcal{C}:\mathcal{C}_1]^{-1}$ in the sense of Definition \ref{def.Markov-inclusion}.
\end{proposition}

\begin{proof}
By symmetry, it suffices to prove that $\cA$ is a projective right $\cAtil_1$-module.

For each $\al \in \Irr(\cC)$, define the subspace $\cA_{\al\cdot \cC_1} \subset \cA$ spanned by all $(i\beta,\beta j)$ with $i,j \in \Irr(\cC)$ and $\beta \in \al \cdot \cC_1$. Note that $\cA_{\al \cdot \cC_1} \subset \cA$ is a right $\cAtil_1$-submodule. As in Lemma \ref{lem:psv-lemma39-generalisation}, denote by $e_{\cC_1}$ the orthogonal projection of $L^2(\cA)$ onto $L^2(\cAtil_1)$. Note that $e_{\cC_1}(\cA) = \cAtil_1$.

Fix $i,\al \in \Irr(\cC)$ and define the projective right $\cAtil_1$-module
$$\cV := \bigoplus_{j \in \Irr(\cC)} \bigl( (i \al , \al j) \ot p_j \cdot \cAtil_1 \bigr) \; .$$
The maps
\begin{align*}
& \theta_1 : p_i \cdot \cA_{\al \cdot \cC_1} \recht \cV : \theta_1(V) = \bigoplus_{j \in \Irr(\cC)} \Bigl(\sum_{W\in\onb(i\alpha,\alpha j)} d(j) W \ot e_{\mathcal{C}_1}(W^\# \cdot V) \Bigr) \; ,\\
& \theta_2 : \cV \recht p_i \cdot \cA_{\al \cdot \cC_1} : \theta_2(W \ot V) = W \cdot V
\end{align*}
are right $\cAtil_1$-linear. By Lemma \ref{lem:psv-lemma39-generalisation}, we have that $\theta_2 \circ \theta_1$ equals a multiple of the identity map on $p_i \cdot \cA_{\al \cdot \cC_1}$. It follows that $p_i \cdot \cA_{\al \cdot \cC_1}$ is a projective right $\cAtil_1$-module.

Taking the (direct) sum over all $i \in \Irr(\cC)$ and over a set of representatives $\alpha_1,\ldots,\alpha_\kappa$ for the $\mathcal{C}_1$-orbits in $\Irr(\mathcal{C})$, we conclude that also $\cA$ is projective as a right $\cAtil_1$-module.

By Lemma \ref{lem:psv-lemma39-generalisation}, we have that
$$\left\{ \; \sqrt{\frac{d(j) \, d(\al_s)}{d([\albar_s \al_s]_{\cC_1})}} \, W \; \middle| \; i,j \in \Irr(\cC), s = 1,\ldots,\kappa, W \in \onb(i\al_s , \al_s j)\; \right\}$$
is a Pimsner-Popa basis for the inclusion $\cAtil_1'' \subset \cA''$. Applying Lemma \ref{lem:psv-lemma39-generalisation} in the case $\cC_1 = \cC$ (and this literally is \cite[Lemma 3.9]{psv-cohom}), we get that
$$\sum_{s = 1}^\kappa \sum_{i,j \in \Irr(\cC)} \sum_{W \in \onb(i\al_s , \al_s j)} \frac{d(j) \, d(\al_s)}{d([\albar_s \al_s]_{\cC_1})} \, W \cdot W^\#
= \sum_{s = 1}^\kappa \sum_{i \in \Irr(\cC)} \frac{d(\al_s)^2}{d([\albar_s \al_s]_{\cC_1})} \, p_i = [\cC : \cC_1] \, 1 \; .$$
So, $\cAtil_1'' \subset \cA''$ is $\lambda$-Markov with $\lambda = [\cC:\cC_1]^{-1}$.
\end{proof}

\begin{proof}[Proof of Theorem \ref{thm:scaling-formula}]
By Lemma \ref{lem.extend-to-Atilde-1}, we have
$$\beta_n^{(2)}(\cC_1) = \dim_{\cAtil_1''} \Tor_n^{\cAtil_1}(L^2(\cAtil_1)^0,\C) \; .$$
By Proposition \ref{prop:tube-algebra-projective-markov}, the left $\cAtil_1$-module $\cA$ is projective. We can thus apply the base change formula for $\Tor$ (see e.g.\ \cite[Proposition~3.2.9]{weibel}) and obtain the isomorphism of left $\cAtil_1''$-modules
$$\Tor_n^{\cAtil_1}(L^2(\cAtil_1)^0,\C) \cong \Tor_n^{\mathcal{A}}(L^2(\widetilde{\mathcal{A}}_1)^0\otimes_{\widetilde{\mathcal{A}}_1} \mathcal{A}, \CC) \; .$$
The left counterpart of Proposition \ref{prop:tube-algebra-projective-markov} provides an inverse for the natural right $\cA$-linear map $L^2(\cAtil_1)^0 \ot_{\cAtil_1} \cA \recht L^2(\cA)^0$, which is thus bijective. We conclude that
$$\Tor_n^{\cAtil_1}(L^2(\cAtil_1)^0,\C) \cong \Tor_n^{\mathcal{A}}(L^2(\cA)^0,\C)$$
as left $\cAtil_1''$-modules.

By Proposition \ref{prop:tube-algebra-projective-markov}, the inclusion $\cAtil_1'' \subset \cA''$ is $\lambda$-Markov with $\lambda = [\cC:\cC_1]^{-1}$ and trace preserving. Using Proposition \ref{prop.dim-scaling-Markov-inclusion}, we conclude that
$$\beta_n^{(2)}(\cC_1) = \dim_{\cAtil_1''} \Tor_n^{\mathcal{A}}(L^2(\cA)^0,\C) = [\cC: \cC_1] \, \dim_{\cA''} \Tor_n^{\mathcal{A}}(L^2(\cA)^0,\C) = [\cC: \cC_1] \,\beta_n^{(2)}(\cC) \; .$$
\end{proof}

Using our results on Markov inclusions, we give the following short proof of Proposition \ref{prop.index-transitive}.

\begin{proof}[{Proof of Proposition \ref{prop.index-transitive}}]
Let $\cC$ be a rigid $C^*$-tensor category with full $C^*$-tensor subcategories $\cC_2 \subset \cC_1 \subset \cC$. Note that $[\cC : \cC_2] < \infty$ if and only if $\Irr(\cC)$ has finitely many $\cC_2$-orbits in the sense of Definition \ref{def:general-index}. Since $\Irr(\cC)$ has finitely many $\cC_2$-orbits if and only if $\Irr(\cC)$ has finitely many $\cC_1$-orbits and $\Irr(\cC_1)$ has finitely many $\cC_2$-orbits, we may assume that the indices $[\cC : \cC_1]$, $[\cC : \cC_2]$ and $[\cC_1: \cC_2]$ are all finite.

Define the $*$-subalgebras $\widetilde{\mathcal{A}}_1 \subset \cA$ and $\widetilde{\mathcal{A}}_2 \subset \cA$ given by \eqref{eqn:a1tilde-def}, associated with $\cC_1 \subset \cC$ and $\cC_2 \subset \cC$, respectively. Note that $\widetilde{\mathcal{A}}_2 \subset \widetilde{\mathcal{A}}_1$. By Proposition \ref{prop:tube-algebra-projective-markov}, the inclusion $\widetilde{\mathcal{A}}_i''\subset\mathcal{A}''$ is $[\cC:\cC_i]^{-1}$-Markov for $i=1,2$. We claim that $\widetilde{\mathcal{A}}_2'' \subset \widetilde{\mathcal{A}}_1''$ is $[\cC_1 : \cC_2]^{-1}$-Markov. This does not literally follow from Proposition \ref{prop:tube-algebra-projective-markov}, but the proof is identical because, choosing representatives $\al_1,\ldots,\al_\kappa$ for the $\cC_2$-orbits in $\Irr(\cC_1)$, Lemma \ref{lem:psv-lemma39-generalisation} implies that
$$\left\{ \; \sqrt{\frac{d(j) \, d(\al_s)}{d([\albar_s \al_s]_{\cC_2})}} \, W \; \middle| \; i,j \in \Irr(\cC), s = 1,\ldots,\kappa, W \in \onb(i\al_s , \al_s j)\; \right\}$$
is a Pimsner-Popa basis for the inclusion $\cAtil_2'' \subset \cAtil_1''$.

Since $\dim_{\cA''}(p_\eps L^2(\cA)) = 1$, a repeated application of Proposition \ref{prop.dim-scaling-Markov-inclusion} gives
\begin{align*}
\dim_{\cAtil_2''}\bigl(p_\eps L^2(\cA)\bigr) & = [\cC : \cC_2] \, \dim_{\cA''}\bigl(p_\eps L^2(\cA)\bigr) = [\cC : \cC_2] \; ,\\
\dim_{\cAtil_2''}\bigl(p_\eps L^2(\cA)\bigr) & = [\cC_1 : \cC_2] \, \dim_{\cAtil_1''}\bigl(p_\eps L^2(\cA)\bigr) = [\cC_1 : \cC_2] \, [\cC : \cC_1] \, \dim_{\cA''}\bigl(p_\eps L^2(\cA)\bigr) \\ &= [\cC : \cC_1] \, [\cC_1 : \cC_2] \; .
\end{align*}
So, the equality $[\cC : \cC_2] = [\cC : \cC_1] \, [\cC_1 : \cC_2]$ is proved.

When $\cC$ has only finitely many irreducible objects and $\cC_1 \subset \cC$ is a full $C^*$-tensor subcategory, we apply this formula to $\cC_2 = \{\eps\}$ and obtain
$$d(\cC) = [\cC : \cC_2] = [\cC : \cC_1] \, [\cC_1 : \cC_2] = [\cC : \cC_1] \, d(\cC_1) \; .$$
So, $[\cC : \cC_1] = d(\cC) / d(\cC_1)$.
\end{proof}

\section{\boldmath $L^2$-Betti numbers for discrete quantum groups}

Following Woronowicz \cite{woronowicz}, a compact quantum group $\bG$ is given by a unital $C^*$-algebra $B$, often suggestively denoted as $B = C(\bG)$, together with a unital $*$-homomorphism $\Delta : B \recht B \otmin B$ to the minimal $C^*$-tensor product satisfying
\begin{itemlist}
\item co-associativity: $(\Delta \ot \id)\Delta = (\id \ot \Delta) \Delta$, and
\item the density conditions: $\Delta(B)(1 \ot B)$ and $\Delta(B) (B \ot 1)$ span dense subspaces of $B \otmin B$.
\end{itemlist}
A compact quantum group $\bG$ admits a unique Haar state, i.e.\ a state $h$ on $B$ satisfying $(\id \ot h)\Delta(b) = (h \ot \id)\Delta(b) = h(b) 1$ for all $b \in B$.

An $n$-dimensional unitary representation $U$ of $\bG$ is a unitary element $U \in M_n(\C) \ot B$ satisfying $\Delta(U_{ij}) = \sum_{k=1}^n U_{ik} \ot U_{kj}$. The category of finite-dimensional unitary representations, denoted as $\Rep(\bG)$, naturally is a rigid $C^*$-tensor category. The coefficients $U_{ij} \in B$ of all finite-dimensional unitary representations of $\bG$ span a dense $*$-subalgebra of $B$, denoted as $\Pol(\bG)$. We have $\Delta(\Pol(\bG)) \subset \Pol(\bG) \ot \Pol(\bG)$, which provides the comultiplication of the Hopf $*$-algebra $\Pol(\bG)$.

The compact quantum group $\bG$ is said to be of Kac type if the Haar state is a trace. This is equivalent with the requirement that for every finite-dimensional unitary representation $U \in M_n(\C) \ot B$, the contragredient $\Ubar \in M_n(\C) \ot B$ defined by $(\Ubar)_{ij} = U_{ij}^*$ is still unitary.

The counit of the Hopf $*$-algebra $\Pol(\bG)$ is the homomorphism $\counit : \Pol(\bG) \recht \C$ given by $\counit(U_{ij}) = 0$ whenever $i \neq j$ and $\counit(U_{ii}) = 1$ for all unitary representations $U \in M_n(\C) \ot B$ of $\bG$.

We denote by $L^2(\bG)$ the Hilbert space completion of $B = C(\bG)$ w.r.t.\ the Haar state $h$. The von Neumann algebra generated by the left action of $B$ on $L^2(\bG)$ is denoted as $L^\infty(\bG)$. The Haar state $h$ extends to a faithful normal state on $L^\infty(\bG)$, which is a trace in the Kac case.

\begin{definition}[{\cite[Definition 1.1]{kyed-quantum-l2hom}}]
Let $\bG$ be a compact quantum group of Kac type. The $L^2$-Betti numbers of the dual discrete quantum group $\bGhat$ are defined as
$$\beta_n^{(2)}(\bGhat) = \dim_{L^\infty(\bG)} \Tor_n^{\Pol(\bG)}(L^2(\bG),\C) \; .$$
\end{definition}

The main result of this section is the following.

\begin{theorem}\label{thm.L2-Betti-quantum-rep-cat}
Let $\bG$ be a compact quantum group of Kac type. Then $\beta_n^{(2)}(\bGhat) = \beta_n^{(2)}(\Rep(\bG))$ for all $n \geq 0$.
\end{theorem}

The equality of $L^2$-Betti numbers in Theorem \ref{thm.L2-Betti-quantum-rep-cat} is surprising. There is no general identification of (co)homology of $\bGhat$ with (co)homology of $\Rep(\bG)$. Indeed, by \cite[Theorem 3.2]{collins-haertel-thom}, homology with trivial coefficients distinguishes between the quantum groups $A_o(k)$, but does not distinguish between their representation categories $\Rep(A_o(k))$ by Corollary \ref{cor.homology-trivial-coefficients} below. Secondly, for the definition of the $L^2$-Betti numbers of a discrete quantum group, the Kac assumption is essential, since we need a trace to measure dimensions. By Theorem \ref{thm.L2-Betti-quantum-rep-cat}, we now also have $L^2$-Betti numbers for non Kac type discrete quantum groups.

\begin{proof}[Proof of Theorem \ref{thm.L2-Betti-quantum-rep-cat}]
Define the $*$-algebra
$$c_c(\bGhat) = \bigoplus_{U \in \Irr(\bG)} M_{d(U)}(\C) \; .$$
Drinfeld's quantum double algebra of $\bG$ is the $*$-algebra $\cA$ with underlying vector space $\Pol(\bG) \ot c_c(\bGhat)$ and product determined as follows. We view $c_c(\bGhat) \subset \Pol(\bG)^*$ in the usual way: the components of $\om \in c_c(\bGhat)$ are given by $\om_{U,ij} = \om(U_{ij})$ for all $U \in \Irr(\bG)$ and $i,j \in \{1,\ldots,d(U)\}$. We write $a \om$ instead of $a \ot \om$ for all $a \in \Pol(\bG)$ and $\om \in c_c(\bGhat)$. The product on $\cA$ is then determined by the following formula:
$$\om \; U_{ij} = \sum_{k,l=1}^n U_{kl} \; \om(U_{ik} \cdot U_{jl}^*)$$
for every unitary representation $U \in M_n(\C) \ot B$. The counit on $\cA$ is given by $\counit(a \om) = \counit(a) \om(1)$ for all $a \in \Pol(\bG)$ and $\om \in c_c(\bGhat) \subset \Pol(\bG)^*$.

Since $\bG$ is of Kac type, the Haar weight $\tau$ on $\cA$ is a trace and it is given by
$$\tau(a \om) = h(a) \sum_{U \in \Irr(\bG)} \sum_{i=1}^{d(U)} d(U) \, \om_{U,ii} \; .$$
We denote by $\cA''$ the von Neumann algebra completion of $\cA$ acting on $L^2(\cA,\tau)$. By \cite[Theorem 2.4]{neshveyev-yamashita-remarks}, the tube algebra of $\Rep(\bG)$ is strongly Morita equivalent with the quantum double algebra $\cA$ defined in the previous paragraph. This strong Morita equivalence respects the counit and the traces on both algebras. Therefore,
$$\beta_n^{(2)}(\Rep(\bG)) = \dim_{\cA''} \Tor_n^{\cA}(L^2(\cA)^0,\C) \; ,$$
where $L^2(\cA)^0$ equals the span of $L^2(\cA) \cdot c_c(\bGhat)$.

On the other hand,
$$\beta_n^{(2)}(\bGhat) = \dim_{L^\infty(\bG)} \Tor_n^{\Pol(\bG)}(L^2(\bG),\C) \; .$$
Since $\cA$ is a free left $\Pol(\bG)$-module, the base change formula for $\Tor$ again applies and gives the isomorphism of left $L^\infty(\bG)$-modules
$$\Tor_n^{\Pol(\bG)}(L^2(\bG),\C) \cong \Tor_n^{\cA}(L^2(\bG) \ot_{\Pol(\bG)} \cA , \C) \; .$$
Since $L^2(\bG) \ot_{\Pol(\bG)} \cA = L^2(\bG) \ot c_c(\bGhat) = L^2(\cA)^0$, we conclude that
$$\beta_n^{(2)}(\bGhat) = \dim_{L^\infty(\bG)} \Tor_n^{\cA}(L^2(\cA)^0,\C) \; .$$
Denoting by $E_{U,ij}$ the natural matrix units for $c_c(\bGhat)$, we see that the elements $\{d(U)^{-1/2} E_{U,ij} \mid U \in \Irr(\bG) , i,j = 1,\ldots, d(U)\}$ form a Pimsner-Popa basis for the (non trace-preserving) inclusion $L^\infty(\bG) \subset \cA''$. It follows that this inclusion is $1$-Markov. Since the left $\cA''$-module $L^2(\cA)^0$ is locally finite (in the sense of Definition \ref{def.locally-finite} and using the example given after Definition \ref{def.locally-finite}), using a bar resolution, one gets that also the left $\cA''$-module $\Tor_n^{\cA}(L^2(\cA)^0,\C)$ is locally finite. Proposition \ref{prop.dim-scaling-Markov-inclusion} then implies that
$$\dim_{L^\infty(\bG)} \Tor_n^{\cA}(L^2(\cA)^0,\C) = \dim_{\cA''} \Tor_n^{\cA}(L^2(\cA)^0,\C)$$
and the theorem is proved.
\end{proof}

Given a compact quantum group $\bG$, all Hopf $*$-subalgebras of $\Pol(\bG)$ are of the form $\Pol(\bH) \subset \Pol(\bG)$, where $\Rep(\bH) \subset \Rep(\bG)$ is a full $C^*$-tensor subcategory. We say that $\Pol(\bH) \subset \Pol(\bG)$ is of finite index if $\Rep(\bH) \subset \Rep(\bG)$ is of finite index in the sense of Definition \ref{def:general-index} and we define the index
$$[\Pol(\bG) : \Pol(\bH)] := [\Rep(\bG) : \Rep(\bH)]$$
using Definition \ref{def:general-index}.

For special types of finite-index Hopf $*$-subalgebras $\Pol(\bH) \subset \Pol(\bG)$, the scaling formula between $\beta_n^{(2)}(\bHhat)$ and $\beta_n^{(2)}(\bGhat)$ was proved in \cite[Theorem D]{bkr16}. Combining Theorems \ref{thm.L2-Betti-quantum-rep-cat} and \ref{thm:scaling-formula}, it holds in general.

\begin{corollary}\label{cor.scaling-compact-quantum}
Let $\bG$ be a compact quantum group of Kac type. Let $\Pol(\bH) \subset \Pol(\bG)$ be a finite-index Hopf $*$-subalgebra. Then,
$$\beta_n^{(2)}(\bHhat) = [\Pol(\bG) : \Pol(\bH)] \; \beta_n^{(2)}(\bGhat) \quad\text{for all}\;\; n \geq 0 \; .$$
\end{corollary}

\begin{remark}
Of course, Corollary \ref{cor.scaling-compact-quantum} can be proven directly, using the same methods as in the proof of Theorem \ref{thm:scaling-formula}. Choosing representatives $U_1,\ldots,U_\kappa$ for the right $\Rep(\bH)$-orbits in $\Irr(\bG)$, the appropriate multiples of $(U_s)_{ij}$ form a Pimsner-Popa basis for the inclusion $L^\infty(\bH) \subset L^\infty(\bG)$. As in the proof of Proposition \ref{prop:tube-algebra-projective-markov}, it follows that $\Pol(\bG)$ is a projective $\Pol(\bH)$-module and that $L^\infty(\bH) \subset L^\infty(\bG)$ is a $\lambda$-Markov inclusion with $\lambda = [\Pol(\bG) : \Pol(\bH)]^{-1}$.
\end{remark}

\section{\boldmath Computing $L^2$-Betti numbers of representation categories}

For any invertible matrix $F\in \mathrm{GL}_m(\CC)$, the \textit{free unitary quantum group} $A_u(F)$ is the universal $C^*$-algebra with generators $U_{ij}$, $1\leq i,j\leq m$, and relations making the matrices $U$ and $F\Ubar F^{-1}$ unitary representations of $A_u(F)$, see \cite{vandaele-wang-uqg}. Here $(\Ubar)_{ij}=(U_{ij})^*$. We denote by $A_u(m)$ the free unitary quantum group given by the $m \times m$ identity matrix. The following is the main result of this section.

\begin{theorem}\label{thm.L2-Betti-Rep-Au}
Let $F\in \mathrm{GL}_m(\CC)$ be an invertible matrix and $\cC = \Rep(A_u(F))$ the representation category of the free unitary quantum group $A_u(F)$. Then,
$$\beta_1^{(2)}(\cC) = 1 \quad\text{and}\quad \beta_n^{(2)}(\cC) = 0 \quad\text{for all}\;\; n \neq 1 \; .$$
\end{theorem}

For $F\in\mathrm{GL}_m(\CC)$ with $F\overline{F}\in\RR 1$, the \textit{free orthogonal quantum group} $A_o(F)$ is the universal $C^*$-algebra with generators $U_{ij}$, $1\leq i,j\leq m$, and relations such that $U$ is unitary and $U=F \Ubar F^{-1}$. We denote by $A_o(m)$ the free orthogonal quantum group given by the $m \times m$ identity matrix. Also note that $\SU_q(2) = A_o\bigl(\begin{smallmatrix} 0 & -q \\ 1 & 0\end{smallmatrix}\bigr)$ for all $q \in [-1,1] \setminus \{0\}$.

Using Theorem \ref{thm.L2-Betti-quantum-rep-cat} in combination with several results of \cite{psv-cohom}, we get the following computations of $L^2$-Betti numbers of discrete quantum groups.

\begin{theorem}\label{thm.L2-Betti-discrete-some-computations}
\begin{enumlist}
\item\label{stat.i} ({\cite[Theorem A]{bkr16} and \cite[Theorem A]{kyed-raum}}) For all $m \geq 2$, we have that $\beta_n^{(2)}(\widehat{A_u(m)})$ is equal to $1$ if $n = 1$ and equal to $0$ if $n \neq 1$.
\item\label{stat.ii} ({\cite[Theorem 1.2]{collins-haertel-thom} and \cite[Corollary 5.2]{vergnioux-paths-L2-cohomology}}) We have that $\beta_n^{(2)}(\widehat{A_o(m)}) = 0$ for all $m \geq 2$ and $n \geq 0$.
\item\label{stat.iii} Let $(B,\tau)$ be a finite-dimensional $C^*$-algebra with its Markov trace. Assume that $\dim B \geq 4$ and let $A_{\text{\rm aut}}(B,\tau)$ be the quantum automorphism group.  Then, $\beta_n^{(2)}(\widehat{A_{\text{\rm aut}}(B,\tau)}) = 0$ for all $n \geq 0$. In particular, all $L^2$-Betti numbers vanish for the duals of the quantum symmetry groups $S_m^+$ with $m \geq 4$.
\item\label{stat.iv} Let $\bG = \bH \wr_* \bF$ be the free wreath product of a nontrivial Kac type compact quantum group $\bH$ and a quantum subgroup $\bF$ of $S_m^+$ that is acting ergodically on $m$ points, $m \geq 2$ (see Remark \ref{rem.necessity} for definitions and comments). Then $\bGhat$ has the same $L^2$-Betti numbers as the free product $\bHhat * \bFhat$, namely
    $$\beta_n^{(2)}(\bGhat) = \begin{cases} \beta_n^{(2)}(\bHhat) + \beta_n^{(2)}(\bFhat) &\;\;\text{if $n \geq 2$,}\\
    \beta_1^{(2)}(\bHhat) + \beta_1^{(2)}(\bFhat) + 1 - (\beta_0^{(2)}(\bHhat)+\beta_0^{(2)}(\bFhat)) &\;\;\text{if $n =1$,}\\
    0 &\;\;\text{if $n =0$.}\end{cases}$$
\item\label{stat.v} In particular, for the duals of the hyperoctahedral quantum group $H_m^+$, $m \geq 4$, and the series of quantum reflection groups $H^{s+}_m$, $s \geq 2$ (see \cite{banica-verngioux-reflection}), all $L^2$-Betti numbers vanish, except $\beta_1^{(2)}$, which is resp.\ equal to $1/2$ and $1-1/s$.
\end{enumlist}
\end{theorem}

\begin{proof}[Proof of Theorem \ref{thm.L2-Betti-Rep-Au}]
By \cite[Theorem~6.2]{bdrv-06}, the rigid $C^*$-tensor category $\Rep(A_u(F))$ only depends on the quantum dimension of the fundamental representation $U$. We may therefore assume that $F\overline{F}=\pm 1$. In \cite[Example~2.18, 3.6]{bny15} and \cite[Proposition~1.2]{bny16}, it is shown that there are exact sequences of Hopf $*$-algebras
\begin{align*}
&\CC\longrightarrow \Pol(\mathbb{H})\longrightarrow \Pol(A_o(F)*A_o(F))\longrightarrow \CC[\ZZ/2\ZZ]\longrightarrow \CC \; ,\\
&\CC\longrightarrow \Pol(\mathbb{H})\longrightarrow \Pol(A_u(F))\longrightarrow \CC[\ZZ/2\ZZ]\longrightarrow \CC \; ,
\end{align*}
for the same compact quantum group $\mathbb{H}$. At the categorical level, this means that $\Rep(A_u(F))$ and the free product $\Rep(A_o(F))*\Rep(A_o(F))$ both contain the same index two subcategory (see also \cite[\S~2]{bkr16}).

By the scaling formula in Theorem \ref{thm:scaling-formula}, this implies that $\Rep(A_u(F))$ and the free product $\Rep(A_o(F))*\Rep(A_o(F))$ have the same $L^2$-Betti numbers. From the free product formula for $L^2$-Betti numbers in \cite[Corollary~9.5]{psv-cohom} and the vanishing of the $L^2$-Betti numbers of $\Rep(A_o(F))$ proved in \cite[Theorem 9.9]{psv-cohom}, the theorem follows.
\end{proof}

\begin{proof}[Proof of Theorem \ref{thm.L2-Betti-discrete-some-computations}]
Using Theorem \ref{thm.L2-Betti-quantum-rep-cat}, \ref{stat.i} follows from Theorem \ref{thm.L2-Betti-Rep-Au} and \ref{stat.ii} follows from \cite[Theorem 9.9]{psv-cohom}. The representation categories of the quantum automorphism groups $A_{\text{\rm aut}}(B,\tau)$ are monoidally equivalent with the natural index $2$ full $C^*$-tensor subcategory of $\Rep(\SU_q(2))$. So \ref{stat.iii} follows from \cite[Theorem 9.9]{psv-cohom} and the scaling formula in Theorem \ref{thm:scaling-formula}.

To prove \ref{stat.iv}, let $\bG = \bH \wr_* \bF$ be a free wreath product as in the formulation of the theorem. We use the notion of Morita equivalence of rigid $C^*$-tensor categories, see \cite[Section 4]{mueger-I} and also \cite[Definition 7.3]{psv-cohom}. By \cite[Theorem B and Remark 7.6]{tarrago-wahl-wreath}, $\Rep(\bG)$ is Morita equivalent in this sense with a free product $C^*$-tensor category $\cC = \cC_1 * \cC_2$ where $\cC_1$ is Morita equivalent with $\Rep(\bH)$ and $\cC_2$ is Morita equivalent with $\Rep(\bF)$. To see this, one uses the observation in \cite[Proposition 9.8]{psv-cohom} that for the Jones tower $N \subset M \subset M_1 \subset \cdots$ of a finite index subfactor $N \subset M$ and for arbitrary intermediate subfactors
$$M_a \subset P \subset M_n \subset M_{n+1} \subset Q \subset M_b$$
with $a \leq n < b$, the $C^*$-tensor category of $P$-$P$-bimodules generated by $P \subset Q$ is Morita equivalent with the $C^*$-tensor category of $N$-$N$-bimodules generated by the original subfactor $N \subset M$. Then, combining \cite[Proposition 7.4 and Corollary 9.5]{psv-cohom}, we find that
$$\beta_n^{(2)}(\bGhat) = \beta_n^{(2)}(\cC_1 * \cC_2) =
\begin{cases} \beta_n^{(2)}(\cC_1) + \beta_n^{(2)}(\cC_2) &\;\;\text{if $n \geq 2$,}\\
    \beta_1^{(2)}(\cC_1) + \beta_1^{(2)}(\cC_2) + 1 - (\beta_0^{(2)}(\cC_1)+\beta_0^{(2)}(\cC_2)) &\;\;\text{if $n =1$,}\\
    0 &\;\;\text{if $n =0$.}\end{cases}$$
Since $\beta_n^{(2)}(\cC_1) = \beta_n^{(2)}(\Rep(\bH)) = \beta_n^{(2)}(\bHhat)$ for all $n \geq 0$, and similarly with $\beta_n^{(2)}(\cC_2)$, statement \ref{stat.iv} is proved.

Finally, by \cite[Theorem 3.4]{banica-verngioux-reflection}, the compact quantum groups $H^{s+}_m$ can be viewed as the free wreath product $(\Z/s\Z) \wr_* S_m^+$ and $H_m^+$ corresponds to the case $s=2$. So \ref{stat.v} follows from \ref{stat.iv}.
\end{proof}

\begin{remark}\label{rem.necessity}
The free wreath products $\bG = \bH \wr_* \bF$ were introduced in \cite{bichon-wreath}. We recall the definition here. Denote by $U \in M_m(\C) \ot C(\bF)$ the fundamental representation of the quantum group $\bF$ acting on $m$ points, so that the action of $\bF$ on $\C^m$ is given by the $*$-homomorphism
$$\al : \C^m \recht \C^m \ot C(\bF) : \al(e_j) = \sum_{i=1}^m e_i \ot U_{ij} \; .$$
Then $C(\bG)$ is defined as the universal $C^*$-algebra generated by $m$ copies of $C(\bH)$, denoted by $\pi_i(C(\bH))$, $i=1,\ldots,m$, together with $C(\bF)$, and the relations saying that $\pi_i(C(\bH))$ commutes with $U_{ij}$ for all $i,j \in \{1,\ldots,m\}$. The comultiplication $\Delta$ on $C(\bG)$ is defined by
$$\Delta(\pi_i(a)) = \sum_{j=1}^m ((\pi_i \ot \pi_j)\Delta(a)) \, (U_{ij} \ot 1) \quad\text{and}\quad \Delta(U_{ij}) = \sum_{k=1}^m U_{ik} \ot U_{kj} \; .$$
Now observe that it is essential to assume in Theorem \ref{thm.L2-Betti-discrete-some-computations}.\ref{stat.iv} that the action of $\bF$ on $\C^m$ is ergodic, in the same way as it is essential to make this hypothesis in \cite[Theorem B]{tarrago-wahl-wreath}. Indeed, in the extreme case where $\bF$ is the trivial one element group, we find that $C(\bH \wr_* \bF)$ is the $m$-fold free product of $C(\bH)$, so that
$$\beta_1^{(2)}(\bGhat) = \beta_1^{(2)}\bigl(\underbrace{\bHhat * \cdots * \bHhat}_{\text{$m$ times}}\bigr) = m(\beta_1^{(2)}(\bHhat) - \beta_0^{(2)}(\bHhat)) + m - 1 \; ,$$
which is different from the value given by Theorem \ref{thm.L2-Betti-discrete-some-computations}.\ref{stat.iv}, namely $\beta_1^{(2)}(\bHhat) - \beta_0^{(2)}(\bHhat)$.
\end{remark}

\begin{remark}
Let $(B,\tau)$ be a finite-dimensional $C^*$-algebra with its Markov trace and assume that $\bF$ is a quantum subgroup of $A_{\text{\rm aut}}(B,\tau)$ that is acting centrally ergodically on $(B,\tau)$. Given any Kac type compact quantum group $\bH$, \cite[Definition 7.5 and Remark 7.6]{tarrago-wahl-wreath} provides an implicit definition of the free wreath product $\bH \wr_* \bF$. The formula in Theorem \ref{thm.L2-Betti-discrete-some-computations}.\ref{stat.iv} remains valid and gives the $L^2$-Betti numbers of $\bH \wr_* \bF$.
\end{remark}

\begin{remark}
The fusion $*$-algebra $\C[\cC]$ of a rigid $C^*$-tensor category $\cC$ has a natural trace $\tau$ and counit $\counit : \C[\cC] \recht \C$ and these coincide with the restriction of the trace and the counit of the tube algebra $\cA$ to its corner $p_\eps \cdot \cA \cdot p_\eps = \C[\cC]$. The GNS construction provides a von Neumann algebra completion $L(\Irr \cC)$ of $\C[\cC]$ acting on $\ell^2(\Irr(\cC))$ and having a natural faithful normal tracial state $\tau$. So also the fusion $*$-algebra $\C[\cC]$ admits $L^2$-Betti numbers defined by
$$\beta_n^{(2)}(\Irr(\cC)) := \dim_{L(\Irr(\cC))} \Tor_n^{\C[\cC]}(\ell^2(\Irr(\cC)),\C) = \dim_{L(\Irr(\cC))} \Ext^n_{\C[\cC]}(\C,\ell^2(\Irr(\cC))) \; .$$
Answering a question posed by Dimitri Shlyakhtenko, we show below that the computation in Theorem \ref{thm.L2-Betti-discrete-some-computations}.\ref{stat.iv} provides the first examples of rigid $C^*$-tensor categories where $\beta_n^{(2)}(\Irr(\cC)) \neq \beta_n^{(2)}(\cC)$. Note that it was already observed in \cite[Comments after Proposition 9.13]{psv-cohom} that for the Temperley-Lieb-Jones category $\cC$, the $C^*$-tensor category and the fusion $*$-algebra have different homology with trivial coefficients.

The first $L^2$-Betti number $\beta_1^{(2)}(\Irr(\cC))$ can be computed as follows. Write $H = \ell^2(\Irr(\cC))$. A linear map $d : \C[\cC] \recht H$ is called a $1$-cocycle if $d(xy) = d(x) \, y + \counit(x) \, d(y)$ for all $x,y \in \C[\cC]$. A $1$-cocycle $d$ is called inner if there exists a vector $\xi \in H$ such that $d(x) = \xi \, x - \counit(x) \, \xi$ for all $x \in \C[\cC]$. Two $1$-cocycles $d_1$ and $d_2$ are called cohomologous if $d_1-d_2$ is inner. The space of $1$-cocycles $Z^1(\C[\cC],H)$ is a left $L(\Irr(\cC))$-module and when $\Irr(\cC)$ is infinite, the subspace of inner $1$-cocycles has $L(\Irr(\cC))$-dimension equal to $1$. In that case, one has
$$\beta_1^{(2)}(\Irr(\cC)) = -1 + \dim_{L(\Irr(\cC))} Z^1(\C[\cC],H) \; .$$

Let $\Gamma$ be any countable group and define $\bG = \Gammahat \wr_* \Z/2\Z$. The fusion rules on $\Irr(\bG)$ were determined in \cite{lemeux-fusion-free-wreath-product} and are given as follows. Denote by $v_1 \in \widehat{\Z / 2 \Z}$ the unique nontrivial element and define $W \subset \Gamma * \widehat{\Z / 2 \Z}$ as the set of reduced words
$$g_0 v_1 g_1 v_1 \cdots v_1 g_{n-1} v_1 g_n \quad , \quad n \geq 0 \; , \; g_0,\ldots,g_n \in \Gamma \setminus \{e\}$$
that start and end with a letter from $\Gamma \setminus \{e\}$. Then $\Irr(\bG)$ can be identified with the set consisting of the trivial representation $v_0$, the one-dimensional representation $v_1$ and a set of $2$-dimensional representations $v(\eps,g,\delta)$ for $\eps,\delta \in \{\pm\}$ and $g \in \Gamma \setminus \{e\}$. The fusion rules are given by
\begin{align*}
& v_1 \ot v(\eps,g,\delta) = v(-\eps,g,\delta) \quad ,\\
& v(\eps,g,\delta) \ot v_1 = v(\eps,g,-\delta) \quad\text{and}\\
& v(\eps,g,\delta) \ot v(\eps',h,\delta') = \begin{cases} v(\eps,gv_1 h, \delta') \oplus v(\eps , gh, \delta') &\;\;\text{if $gh \neq e$,}\\
v(\eps,gv_1 h, \delta') \oplus v_1 \oplus v_0 &\;\;\text{if $gh = e$.}\end{cases}
\end{align*}
Write $H = \ell^2(\Irr(\bG))$. Given an arbitrary family of vectors $(\xi_g)_{g \in \Gamma \setminus \{e\}}$ in $H$, one checks that there is a uniquely defined $1$-cocycle $d : \C[\cC] \recht H$ satisfying $d(v_0) = d(v_1) = 0$ and $d(v(\eps,g,\delta)) = \xi_g$ for all $g \in \Gamma \setminus \{e\}$ and $\eps,\delta \in \{\pm\}$. Moreover, this provides exactly the $1$-cocycles that vanish on $v_0$ and $v_1$.
Every $1$-cocycle is cohomologous to a $1$-cocycle vanishing on $v_0$, $v_1$, and the inner $1$-cocycles vanishing on $v_0$, $v_1$ have $L(\Irr(\bG))$-dimension $1/2$. It follows that
$$\beta_1^{(2)}(\Irr(\bG)) = |\Gamma| - 1 - \frac{1}{2} = |\Gamma|-\frac{3}{2} \; .$$
On the other hand, by Theorem \ref{thm.L2-Betti-discrete-some-computations}.\ref{stat.iv}, we have
$$\beta_1^{(2)}(\Rep(\bG)) = \beta_1^{(2)}(\Gamma) - \beta_0^{(2)}(\Gamma) + \frac{1}{2} \; .$$
Taking $\Gamma = \Z$, we find an example where $\beta_1^{(2)}(\Irr(\bG)) = \infty$, while $\beta_1^{(2)}(\Rep(\bG))=1/2$. Taking $\Gamma = \Z / 2\Z$, we find an example where $\Rep(\bG)$ is an amenable $C^*$-tensor category, but yet $\beta_1^{(2)}(\Irr(\bG)) = 1/2 \neq 0$. Although amenability can be expressed as a property of the fusion rules together with the counit (which provides the dimensions of the irreducible objects), amenability does not ensure that the fusion $*$-algebra has vanishing $L^2$-Betti numbers. In particular, the Cheeger-Gromov argument given in \cite[Theorem 8.8]{psv-cohom} does not work on the level of the fusion $*$-algebra. In the above example, $\Rep(\bG)$ is Morita equivalent to the group $\Gamma * \Z/2\Z$. So also invariance of $L^2$-Betti numbers under Morita equivalence does not work on the level of the fusion $*$-algebra. All in all, this illustrates that it is not very natural to consider $L^2$-Betti numbers for fusion algebras.
\end{remark}

\section{Projective resolution for the Temperley-Lieb-Jones category}

Fix $q \in [-1,1] \setminus \{0\}$ and realize the Temperley-Lieb-Jones category $\cC$ as the representation category $\cC = \Rep(\SU_q(2))$. Denote by $\cA$ the tube algebra of $\cC$ together with its counit $\counit : \cA \recht \C$.

Although it was proved in \cite[Theorem 9.9]{psv-cohom} that $\beta_n^{(2)}(\cC) = 0$ for all $n \geq 0$, an easy projective resolution of $\counit : \cA \recht \C$ was not given in \cite{psv-cohom}. On the other hand, \cite[Theorem 5.1]{bichon-yetter-drinfeld} provides a length 3 projective resolution for the counit of $\Pol(\SU_q(2))$. In the case of $\Pol(A_o(m))$, this projective resolution was already found in \cite[Theorem 1.1]{collins-haertel-thom}, but the proof of its exactness was very involved and ultimately relied on a long, computer-assisted Gr\"{o}bner base calculation. The proof in \cite{bichon-yetter-drinfeld} is much simpler and moreover gives a resolution by so-called Yetter-Drinfeld modules. This means that it is actually a length 3 projective resolution for the quantum double algebra of $\SU_q(2)$. By \cite[Theorem 2.4]{neshveyev-yamashita-remarks}, this quantum double algebra is strongly Morita equivalent with the tube algebra $\cA$. The following is thus an immediate consequence of \cite[Theorem 5.1]{bichon-yetter-drinfeld}.

\begin{theorem}\label{thm.resolution-TLJ}
Label by $(v_n)_{n \in \N}$ the irreducible objects of $\cC = \Rep(\SU_q(2))$ and denote by $(p_n)_{n \in \N}$ the corresponding projections in $\cA$.

Decomposing $v_1 v_1 = v_0 \oplus v_2$, the identity operator $1 \in ((v_1 v_1)v_1,v_1 (v_1 v_1))$ defines a unitary element $V \in (p_0 + p_2)\cdot \cA \cdot (p_0 + p_2)$. Denoting by $\tau \in \{\pm 1\}$ the sign of $q$, the sequence
$$0 \recht \cA \cdot p_0 \overset{W \mapsto W \cdot p_0 \cdot (V+\tau)}{\longrightarrow} \cA \cdot (p_0 + p_2) \overset{W \mapsto W \cdot (V-\tau)}{\longrightarrow} \cA \cdot (p_0 + p_2) \overset{W \mapsto W \cdot (V+\tau) \cdot p_0}{\longrightarrow} \cA \cdot p_0 \overset{\counit}{\recht} \C$$
is a resolution of $\counit : \cA \recht \C$ by projective left $\cA$-modules.
\end{theorem}

As a consequence of Theorem \ref{thm.resolution-TLJ}, we immediately find the (co)homology of $\cC = \Rep(\SU_q(2))$ with trivial coefficients $\C$, which was only computed up to degree $2$ in \cite[Proposition 9.13]{psv-cohom}. The same result was found in an unpublished note of Y.~Arano using different methods.

\begin{corollary}\label{cor.homology-trivial-coefficients}
For $\cC = \Rep(\SU_q(2))$, the homology $H_n(\cC,\C)$ and cohomology $H^n(\cC,\C)$ with trivial coefficients are given by $\C$ when $n=0,3$ and are $0$ when $n \not\in \{0,3\}$.
\end{corollary}

\begin{remark}
It is straightforward to check that inside $\cA$, we have $p_0 \cdot V \cdot V = p_0$ and $\counit(V) = - \tau$. Therefore, the composition of two consecutive arrows in Theorem \ref{thm.resolution-TLJ} indeed gives the zero map. Using the diagrammatic representation of the tube algebra $\cA$ given in \cite[Section 5.2]{ghosh-jones}, there are natural vector space bases for $\cA \cdot p_0$ and $\cA \cdot (p_0 + p_2)$. It is then quite straightforward to check that the sequence in Theorem \ref{thm.resolution-TLJ} is indeed exact.

Using the same bases, one also checks that the tensor product of this resolution with $L^2(\cA) \ot_{\cA} \cdot$ stays dimension exact. This then provides a slightly more elementary proof that $\beta_n^{(2)}(\cC) = 0$ for all $n \geq 0$, as was already proved in \cite[Theorem 9.9]{psv-cohom}.
\end{remark}

\begin{remark}
Section 9.5 of \cite{psv-cohom} provides a diagrammatic complex to compute the homology $H_n(\cC,\C)$ with trivial coefficients. In the particular case where $\cC$ is the Temperley-Lieb-Jones category $\TLJ(\delta)=\Rep(\SU_q(2))$ with $-1<q<0$ and $\delta = -q - 1/q$, the space of $n$-chains is given by the linear span of all configurations of nonintersecting circles embedded into the plane with $n$ points removed. Using Theorem \ref{thm.resolution-TLJ}, one computes that the $3$-homology is spanned by the $3$-cycle
$$c_1 = \quad \begin{tikzpicture}[baseline={([yshift=6pt]current bounding box.south)},scale=.3]
        \draw[fill] (1,1) circle (3pt);
        \draw[fill] (3,1) circle (3pt);
        \draw[fill] (5,1) circle (3pt);
        \draw[line width=1pt] (0.25,1) .. controls (0.25,0) and (1.75,0) .. (1.75,1) .. controls (1.75,2) and (4.25,2) .. (4.25,1) .. controls (4.25,0) and (5.75,0) .. (5.75,1) .. controls (5.75,3) and (0.25,3) .. (0.25,1);
    \end{tikzpicture} \quad + \quad
    \begin{tikzpicture}[baseline={([yshift=5pt]current bounding box.south)},scale=.3]
    	\draw[fill] (1,1) circle (3pt);
        \draw[fill] (3,1) circle (3pt);
        \draw[fill] (5,1) circle (3pt);
        \draw[line width=1pt] (4,1) ellipse (50pt and 25pt);
    \end{tikzpicture}
$$
It is however less clear how to write effectively a generating $3$-cocycle in this diagrammatic language. For instance, for every integer $k \geq 1$, indicating by $k$ the number of parallel strings, also
$$c_k = \quad \begin{tikzpicture}[baseline={([yshift=6pt]current bounding box.south)},scale=.3]
        \draw[fill] (1,1) circle (3pt);
        \draw[fill] (3,1) circle (3pt);
        \draw[fill] (5,1) circle (3pt);
        \draw[line width=1.5pt] (0.25,1) .. controls (0.25,0) and (1.75,0) .. (1.75,1) .. controls (1.75,2) and (4.25,2) .. (4.25,1) .. controls (4.25,0) and (5.75,0) .. (5.75,1) .. controls (5.75,3) and (0.25,3) .. (0.25,1);
        \node at (5.5,2.8) {$k$};
    \end{tikzpicture} \quad + \quad
    \begin{tikzpicture}[baseline={([yshift=5pt]current bounding box.south)},scale=.3]
    	\draw[fill] (1,1) circle (3pt);
        \draw[fill] (3,1) circle (3pt);
        \draw[fill] (5,1) circle (3pt);
        \draw[line width=1.5pt] (4,1) ellipse (50pt and 25pt);
        \node at (5.5,2.4) {$k$};
    \end{tikzpicture}
$$
and
$$d_k = \quad \begin{tikzpicture}[baseline={([yshift=8mm]current bounding box.south)},scale=.3]
\draw[fill] (1,1) circle (3pt);
\draw[fill] (5,1) circle (3pt);
\draw[fill] (9,1) circle (3pt);
\draw[line width=1.5pt] (0.25,1) .. controls (0.25,0) and (1.75,0) .. (1.75,1) .. controls (1.75,2) and (5.75,2) .. (5.75,1) .. controls (5.75,-1) and (-0.75,-1) .. (-0.75,1) .. controls (-0.75,4) and (8.25,4) .. (8.25,1) .. controls (8.25,0) and (9.75,0) .. (9.75,1) .. controls (9.75,5) and (-2.25,5) .. (-2.25,1) .. controls (-2.25,-2) and (7.25,-2) .. (7.25,1)  .. controls (7.25,3) and (0.25,3) .. (0.25,1);
\node at (9,3.6) {$k$};
\end{tikzpicture} \quad + \quad
    \begin{tikzpicture}[baseline={([yshift=5pt]current bounding box.south)},scale=.3]
    	\draw[fill] (1,1) circle (3pt);
        \draw[fill] (3,1) circle (3pt);
        \draw[fill] (5,1) circle (3pt);
        \draw[line width=1.5pt] (4,1) ellipse (50pt and 25pt);
        \node at (5.5,2.4) {$k$};
    \end{tikzpicture}
$$
are $3$-cycles and ad hoc computations show that in $3$-homology, we have $c_k = k \delta^{k-1} \, c_1$ and $d_k = 3k\delta^{k-1} \, c_1$. It would be interesting to have a geometric procedure to identify a given $3$-cycle with a multiple of $c_1$ and to prove geometrically that homology vanishes in higher degrees.


\end{remark}

\section{\boldmath Derivations on rigid $C^*$-tensor categories} \label{sec.derivations-Cstar-tensor}

\subsection{A Drinfeld type central element in the tube algebra}

To describe the first cohomology of a rigid $C^*$-tensor category $\cC$ by a space of derivations, a natural element in the center of the tube algebra (more precisely, in the center of its multiplier algebra) plays a crucial role. In the case where $\cC$ has only finitely many irreducible objects and hence, the tube algebra $\cA$ is a direct sum of matrix algebras, this Drinfeld type central element was introduced in \cite[Theorem 3.3]{izumi-sectors-I}. When $\cC$ has infinitely many irreducible objects, the same definition applies and yields the following central unitary $U$ in the multiplier algebra $M(\cA)$ defined by unitary elements $U_i \in p_i \cdot \cA \cdot p_i$.

Fix a rigid $C^*$-tensor category $\cC$. For every $i \in \Irr(\cC)$, denote by $U_i \in p_i \cdot \cA \cdot p_i$ the element defined by the identity map in $(ii,ii)$.

\begin{proposition}
\label{thm:ui-properties}
Fix $i,j\in\Irr(\mathcal{C})$.
Then $U_i^\#=s_it_i^*$ and $U_i\cdot U_i^\#=U_i^\#\cdot U_i=p_i$.
In other words, $U_i$ is unitary in $p_i\cdot\mathcal{A}\cdot p_i$.
Moreover, for any $\alpha\in\Irr(\mathcal{C})$ and $V\in (i\alpha,\alpha j)$, the following relation holds:
\begin{equation}
\label{eqn:ui-commutation}
U_i\cdot V =V\cdot U_j =\sum_{\gamma\in\Irr(\mathcal{C})}d(\gamma)\sum_{\substack{W\in\onb(i\alpha,\gamma) W'\in\onb(\alpha j,\gamma)}} \langle V, WW'^*\rangle \; (1 \otimes W'^*)(W\otimes 1) \; .
\end{equation}
So, $U := \sum_{i \in I} U_i$ is a central unitary element in the multiplier algebra $M(\cA)$.
\end{proposition}

\begin{proof}
By definition of the involution on $\mathcal{A}$, we have that
\[
U_i^\#=(t_i^*\otimes 1 \otimes 1)(1 \otimes 1 \otimes s_i)=s_it_i^*\in (i\overline{i},\overline{i}i) \; .
\]
Given this, one finds that
\begin{align*}
U_i^\#\cdot U_i&=\sum_{\gamma\in\Irr(\mathcal{C})}d(\gamma)\sum_{W\in\onb(\overline{i}i,\gamma)} (1 \otimes W^*)(U_i^\#\otimes 1)(1 \otimes U_i)(W\otimes 1)\\
&=\sum_{\gamma\in\Irr(\mathcal{C})}d(\gamma)\sum_{W\in\onb(\overline{i}i,\gamma)} (1 \otimes W^*)(s_it_i^*W\otimes 1) \; .
\end{align*}
Note that all terms with $\gamma\neq\eps$ vanish.
Hence, to conclude the computation, it suffices to note that $\{d(i)^{-1/2}t_i\}$ is an orthonormal basis for $(\overline{i}i,\eps)$.
Similarly, one checks that $U_i\cdot U_i^\#=p_i$.

Choose $V\in p_i\cdot\mathcal{A}\cdot p_j$ arbitrarily.
Then
\begin{align*}
U_i\cdot V &= \sum_{\gamma\in\Irr(\mathcal{C})}\sum_{W\in\onb(i\alpha,\gamma)} d(\gamma) (1 \otimes W^* V)(W\otimes 1) \; .
\end{align*}
On the other hand,
\begin{align*}
V\cdot U_j &= \sum_{\gamma\in\Irr(\mathcal{C})}\sum_{W'\in\onb(j\alpha,\gamma)} d(\gamma) (1 \otimes W'^*)(VW'\otimes 1) \; .
\end{align*}
From these identities, one readily deduces \eqref{eqn:ui-commutation}, by expanding $W^*V$ (resp.\@ $VW'$) in terms of the other orthonormal basis and using that the scalar products are given by the categorical traces.
\end{proof}

Note that \eqref{eqn:ui-commutation}, along with the fact that $U_\eps=p_\eps$ in particular implies that
\begin{equation}
\label{eqn:ui-eps-absorb}
U_i\cdot V = V \text{\qquad and\qquad} W\cdot U_i = W
\end{equation}
for $V\in p_i\cdot\mathcal{A}\cdot p_\eps$ and $W\in p_\eps\cdot \mathcal{A}\cdot p_i$.
As another corollary of \eqref{eqn:ui-commutation}, we find that $U=\sum_{i\in\Irr(\mathcal{C})} U_i$ belongs to the center of the von Neumann algebra $\mathcal{A}''$.

\subsection{\boldmath Properties of $1$-cocycles}

Let $\cC$ be a rigid $C^*$-tensor category with tube algebra $\mathcal{A}$. Fix a nondegenerate right Hilbert $\cA$-module $\cK$.
As in \cite{psv-cohom}, define the bar complex for Hochschild (co)homology as follows.
Denote by $\mathcal{B}$ the linear span of the projections $p_i$, $i \in \Irr(\cC)$. Then define
\[
C_n=p_\eps\cdot\mathcal{A}\otimes_{\mathcal{B}}\underbrace{\mathcal{A}\otimes_{\mathcal{B}}\cdots\otimes_{\mathcal{B}}\mathcal{A}}_{n\text{ factors }}
\]
with boundary maps $\partial:C_n\to C_{n-1}: \sum_{k=0}^n (-1)^k\partial_k$ where
\[
\partial_k(V_0\otimes\cdots\otimes V_n)=\begin{cases}
\counit(V_0)p_\eps\cdot V_1\otimes\cdots\otimes V_n & k=0,\\
V_0\otimes\cdots\otimes V_{k-1}\cdot V_k\otimes\cdots\otimes V_n & 1\leq k\leq n \; .
\end{cases}
\]
This is a resolution of the trivial right $\cA$-module $\C$ by projective right $\mathcal{A}$-modules. So $H^n(\mathcal{C},\mathcal{K}^0)$ is the $n$-th cohomology of the dual complex
\begin{equation}
\label{eqn:original-l2-bar-complex}
 \Hom_{\mathcal{A}}(C_n,\mathcal{K}^0)=\Hom_{\mathcal{A}}(p_\eps\cdot\mathcal{A}\otimes_{\mathcal{B}}\underbrace{\mathcal{A}\otimes_{\mathcal{B}}\cdots\otimes_{\mathcal{B}}\mathcal{A}}_{n\text{ factors }}\,,\mathcal{K}^0).
\end{equation}

The complex in \eqref{eqn:original-l2-bar-complex} is isomorphic with the complex
\[
\tilde{C}^{n}=\Hom_{\mathcal{B}}(p_\eps\cdot\mathcal{A}\otimes_{\mathcal{B}}\underbrace{\mathcal{A}\otimes_{\mathcal{B}}\cdots\otimes_{\mathcal{B}}\mathcal{A}}_{n-1\text{ factors }}\,,\mathcal{K}^0)
\]
where $\tilde{C}^0=\mathcal{K}\cdot p_\eps$. For $n\geq 1$, the coboundary maps of this complex are given by $\partial: \tilde{C}^{n}\to\tilde{C}^{n+1}: \sum_{k=0}^{n+1}(-1)^k\partial_k$ where
\[
\partial_k(D)(V_0\otimes \cdots \otimes V_n)
=\begin{cases}
\counit(V_0) D\left(p_\eps\cdot V_1\otimes \cdots \otimes V_n\right) & k=0,\\
D\left(V_0\otimes \cdots \otimes V_{k-1}\cdot V_k\otimes \cdots\otimes V_n\right) & 1\leq k\leq n,\\
D\left(V_0\otimes\cdots \otimes V_{n-1}\right)\cdot V_n & k = n+1.
\end{cases}
\]
The zeroth coboundary map of $\tilde{C}^\bullet$ is given by
\begin{equation}
\label{eqn:inner-cocycle-def}
            \mathcal{K}\cdot p_\eps \to \Hom_{\mathcal{B}}(p_\eps\cdot\mathcal{A},\mathcal{K}^0): \xi\mapsto \left[D_\xi: V\mapsto \counit(V)\xi-\xi\cdot V\right].
\end{equation}

In this picture, the 1-cocycles are precisely the maps $D\in\Hom_{\mathcal{B}}(p_\eps\cdot\mathcal{A},\mathcal{K}^0)$ that satisfy
\begin{equation}
\label{eqn:cocycle-relation}
D(V\cdot W)=D(V)\cdot W+\counit(V)D(W)
\end{equation}
for all $V\in p_\eps\cdot\mathcal{A}\cdot p_i$ and $W\in p_i\cdot\mathcal{A}\cdot p_j$.
We associate a cocycle $D_\xi$ to every vector $\xi\in\mathcal{K}\cdot p_\eps$ via \eqref{eqn:inner-cocycle-def}.
These are the \textit{inner} $1$-cocycles.

By analogy with the first $L^2$-Betti number for groups, we want to express how a 1-cocycle $D$ is determined by its values on a generating set of objects of $\mathcal{C}$. So, we first need to specify how $D$ can actually be evaluated on objects $\alpha\in\mathcal{C}$.

By the correspondence theorem from \cite{psv-cohom} discussed in Section~\ref{sec:representation-theory}, we may suppose that the right Hilbert $\cA$-module $\mathcal{K}$ arises from a unitary half braiding $(X,\sigma)\in\mathcal{Z}(\ind{\mathcal{C}})$, where $X \in \ind{\cC}$ satisfies $(X,i) = \cK \cdot p_i$ for all $i \in \Irr(\cC)$.

For every $\al \in \Irr(\cC)$, we consider the vector subspace $\cA_\al \subset \cA$
\begin{equation}\label{eq.subspace-A-al}
\cA_\al = \bigoplus_{i,j \in \Irr(\cC)} (i \al , \al j) \; .
\end{equation}
Note that each $\cA_\al$ is a $\cB$-bimodule. We can then define the natural bijection
$$\Hom_\cB(p_\eps \cdot \cA_\al , \cK^0) \cong (\al X, \al)$$
identifying $D \in \Hom_\cB(p_\eps \cdot \cA_\al , \cK^0)$ with $D_\al \in (\al X, \al)$ through the formulae
$$D(V) = (\Tr_\al \ot \id)(D_\al V) \quad\text{and}\quad D_\al = \sum_{j \in \Irr(\cC)} \sum_{W \in \onb(\al,\al j)} d(j) (1 \ot D(W))W^*$$
for all $i \in \Irr(\cC)$ and $V \in (\al,\al i)$.
Putting all $\al \in \Irr(\cC)$ together, we find a bijection
$$\Hom_\cB(p_\eps \cdot \cA , \cK^0) \cong \prod_{\al \in \Irr(\cC)} (\al X, \al)$$
identifying $D \in \Hom_\cB(p_\eps \cdot \cA , \cK^0)$ with the family $(D_\al)_{\al \in \Irr(\cC)}$.

Given a family of elements $D_\al \in (\al X, \al)$ for all $\al \in \Irr(\cC)$, we uniquely define $D_\beta \in (\beta X , \beta)$ for arbitrary objects $\beta \in \cC$ by the formula
$$D_\beta = \sum_{\al \in \Irr(\cC)} \sum_{V \in \onb(\al,\be)} d(\al) (V^* \ot 1) D_\al V \; .$$
Note that the naturality condition
\begin{equation}\label{eqn:dalpha-naturality}
D_\al V = (V \ot 1) D_\be
\end{equation}
holds for all $\al,\be \in \cC$ and all $V \in (\al,\be)$.

\begin{definition}
Let $\mathcal{C}$ be a rigid $C^*$-tensor category. We say that a subset $\cG \subset \Irr(\cC)$ \emph{generates} $\cC$ when every irreducible object in $\mathcal{C}$ arises as a subobject of some tensor product of elements in $\mathcal{G} \cup \overline{\mathcal{G}}$.
\end{definition}

The following proposition implies that a $1$-cocycle $D \in \Hom_\cB(p_\eps \cdot \cA, \cK^0)$ is completely determined by its ``values'' $D_\al \in (\al X,\al)$ for $\al$ belonging to a generating set $\cG \subset \Irr(\cC)$.

\begin{proposition}\label{prop:cocycle-dalpha-properties}
Consider a morphism $D\in\Hom_{\mathcal{B}}(p_\eps\cdot\mathcal{A},\mathcal{K}^0)$ with corresponding values $D_\al \in (\al X, \al)$, $\al \in \cC$. Then $D$ is a 1-cocycle if and only if
\begin{equation}
\label{eqn:cocycle-dalpha-charact}
D_{\alpha\beta}=(1 \otimes\sigma_\beta^*)(D_\alpha\otimes 1)+(1 \otimes D_\beta)
\end{equation}
for all $\alpha,\beta\in\mathcal{C}$. In particular, any 1-cocycle $D$ satisfies $D_\eps=0$ and
\begin{align}
D_{\overline{\alpha}}&=-\sigma^*_{\overline{\alpha}}(t_\alpha^*\otimes 1 \ot 1)(1 \otimes D_\alpha\otimes 1)(1 \otimes s_\alpha)\label{eqn:dalpha-adjoint-constraint}\\
&=-(1 \otimes s_\alpha^*\otimes 1)(1 \ot 1 \otimes\sigma_{\overline{\alpha}}^*)(1 \otimes D_\alpha\otimes 1)(t_\alpha\otimes 1) \notag
\end{align}
for all $\al \in \Irr(\cC)$.
\end{proposition}

\begin{proof}
Choose arbitrary morphisms  $V\in (\alpha,\alpha i)$ and $W\in (i\beta,\beta j)$.
The following identities can be verified by direct computation:
\begin{align*}
D(V\cdot W) &= (\Tr_{\alpha\beta}\otimes\id)(D_{\alpha\beta}(V\otimes 1)(1 \otimes W)) \; ,\\
D(V)\cdot W &= (\Tr_{\alpha\beta}\otimes\id)\left((1 \otimes\sigma_\beta^*)(D_\alpha\otimes 1)(V\otimes 1)(1 \otimes W)\right) \; ,\\
\counit(V)D(W)&=(\Tr_{\alpha\beta}\otimes\id)((1 \otimes D_\beta)(V\otimes 1)(1 \otimes W)) \; .
\end{align*}
By Frobenius reciprocity, for every fixed $\al,\be, j \in \Irr(\cC)$, the linear span of all $(V \ot 1)(1 \ot W)$ with $i \in \Irr(\cC)$, $V \in (\al,\al i)$, $W \in (i \be, \be j)$ equals $(\al \be , \al \be j)$. So it follows that $D$ is a $1$-cocycle if and only if \eqref{eqn:cocycle-dalpha-charact} holds for all $\al,\be \in \cC$.

Finally, assume that $D$ is a $1$-cocycle. By \eqref{eqn:cocycle-dalpha-charact}, we get that $D_\eps=0$.
The naturality property of the $D_\alpha$ implies that $D_{\alpha\overline{\alpha}}s_\alpha=0$ for all $\alpha\in\mathcal{C}$. So,
\[
(1 \ot 1 \otimes D_{\overline{\alpha}})(1 \otimes s_\alpha)=-(1 \ot 1 \otimes\sigma^*_{\overline{\alpha}})(1 \otimes D_\alpha\otimes 1)(1 \otimes s_\alpha) \; ,
\]
which yields one half of \eqref{eqn:dalpha-adjoint-constraint} after multiplying by $(t_\alpha^*\otimes 1)$ on both sides.
The other identity is proven similarly, by observing that $(s_\alpha^*\otimes 1)D_{\alpha\overline{\alpha}}=0$.
\end{proof}

The following lemma shows that the constraint \eqref{eqn:dalpha-adjoint-constraint} on $D_\alpha$ can be succinctly restated in terms of the special unitaries $U_i\in\mathcal{A}$ introduced in the previous section.

\begin{lemma}\label{lem:cocycle-necessary-conditions}
Fix $\al \in \Irr(\cC)$ and consider $\cA_\al \subset \cA$ as in \eqref{eq.subspace-A-al}. Let $D \in \Hom_\cB(p_\eps \cdot \cA_\al,\cK^0)$ with corresponding $D_\al \in (\al X, \al)$. Then $D_\al$ satisfies the relation
\begin{equation}\label{eqn:dalpha-complicated}
\sigma^*_{\overline{\alpha}}(t_\alpha^* \otimes 1 \ot 1)(1 \otimes D_\al \otimes 1)(1 \otimes s_\alpha) =(1 \otimes s_\alpha^*\otimes 1)(1 \ot 1 \otimes\sigma_{\overline{\alpha}}^*)(1 \otimes D_\al \otimes 1)(t_\alpha\otimes 1)
\end{equation}
if and only if $D(V) = D(V) \cdot U_i$ for all $i \in \Irr(\cC)$ and $V \in (\al,\al i)$.
\end{lemma}

\begin{proof}
By definition of the $\cA$-module structure on $\mathcal{K}$, for every $i \in \Irr(\cC)$ and $V \in (\al,\al i)$, we have
\begin{align*}
D(V)\cdot U_i &= (\Tr_\alpha\otimes\id)(D_\al V)\cdot U_i =(\Tr_{\alpha i} \otimes \id)((1 \otimes \sigma_i^*)(D_\al V\otimes 1))\\
&=(\Tr_\alpha\otimes\id)((V\otimes 1)(1 \otimes\sigma_i^*)(D_\al \otimes 1)) =(\Tr_\alpha\otimes\id)((V\otimes 1)\sigma_{\alpha i}^* (\sigma_\alpha D_\al\otimes 1))\\
&=(\Tr_\alpha\otimes\id)(\sigma_\alpha^*(1 \ot V)(\sigma_\alpha D_\al \otimes 1)) \; ,
\end{align*}
where the final two equalities follow from the half braiding property and the naturality of $\sigma$, respectively. Writing $V$ as $V=(s_\alpha^*\otimes 1)(1 \otimes W)$ with $W\in (\overline{\alpha}\alpha,i)$, we then find that
\begin{align*}
D(V)\cdot U_i &= (\Tr_\alpha\otimes\id)(\sigma_\alpha^*(1 \otimes s_\alpha^*\otimes 1)(1 \ot 1 \otimes W)(\sigma_\alpha D_\al \otimes 1)) \\
&= (\Tr_\alpha\otimes\id)(\sigma_\alpha^*(1 \otimes s_\alpha^*\otimes 1)(\sigma_\alpha D_\al \otimes 1 \ot 1))W \; .
\end{align*}
Since $D(V) = (\Tr_\al \ot \id)(D_\al V) = (t_\al^* \ot 1)(1 \ot D_\al) W$, we conclude that the equality $D(V) = D(V) \cdot U_i$ for all $i \in \Irr(\cC)$ and $V \in (\al,\al i)$ is equivalent with the equality
$$
(t_\alpha^* \otimes 1)(1 \otimes D_\al) =(t_\alpha^*\otimes 1)(1 \otimes\sigma_\alpha^*)(1 \ot 1 \otimes s_\alpha^*\otimes 1)(1 \otimes\sigma_\alpha D_\al \otimes 1 \ot 1)(t_\alpha\otimes 1 \ot 1) \; .$$
Applying the transformation $Y \mapsto \sigma_{\overline{\alpha}}^* (Y \ot 1)(1 \ot s_\al)$ to the left- and the right-hand side, this equality becomes equivalent with \eqref{eqn:dalpha-complicated}.
\end{proof}

We can then formalize how a $1$-cocycle is determined by its values on a generating set of a rigid $C^*$-tensor category as follows.

\begin{proposition}\label{prop:cocycle-upper-bound}
Let $\cC$ be a rigid $C^*$-tensor category with finite generating set $\cG \subset \Irr(\cC)$. Denote by $\cA$ the tube algebra of $\cC$ and let $\cK$ be a nondegenerate right Hilbert $\cA$-module. For every $i \in \Irr(\cC)$, define the subspace $\cKfix_i \subset \cK \cdot p_i$ given by
$$\cKfix_i := \{\xi \in \cK \cdot p_i \mid \xi \cdot U_i = \xi \} \; .$$
Define
\[
\tilde{Z}^1(\mathcal{C},\mathcal{K}^0)=\bigoplus_{\al \in \cG} \bigoplus_{i\in\Irr(\mathcal{C})} \cKfix_i \otimes (\alpha i,\alpha) \; .
\]
Then, the linear map
\[
\Phi:Z^1(\mathcal{C},\mathcal{K}^0)\to \tilde{Z}^1(\mathcal{C},\mathcal{K}^0): D\mapsto \bigoplus_{\al \in \cG} \bigoplus_{i\in\Irr(\mathcal{C})} \sum_{W\in\onb(\alpha,\alpha i)} D(W)\otimes W^*
\]
is injective. In particular, if $\cC$ has infinitely many irreducible objects, we find the estimate
\begin{equation}
\label{eqn:dimension-upper-bound}
\beta_1^{(2)}(\cC) \leq -1 + \sum_{\al \in \cG} \sum_{i\in\Irr(\mathcal{C})} \mult(i,\overline{\alpha}\alpha) \; \tau(q_i) \; ,
\end{equation}
where $q_i \in p_i \cdot \cA'' \cdot p_i$ denotes the projection onto the kernel of $U_i - p_i$.
\end{proposition}

\begin{proof}
By Proposition \ref{prop:cocycle-dalpha-properties} and Lemma \ref{lem:cocycle-necessary-conditions}, the map $\Phi$ is well defined and injective. In the case where $\cK = L^2(\cA)$, the map $\Phi$ is left $\cA''$-linear. Since $\dim_{\cA''} L^2(\cA) \cdot q_i = \tau(q_i)$, the proposition follows once we have proved that the space of inner $1$-cocycles has $\cA''$-dimension equal to $1$, assuming that $\cC$ has infinitely many irreducible objects.

In that case, $\beta_0^{(2)}(\cC) = 0$ by \cite[Corollary 9.2]{psv-cohom}, meaning that the coboundary map
$$L^2(\cA) \cdot p_\eps \recht \Hom_\cB(p_\eps \cdot \cA , L^2(\cA)^0)$$
is injective. The space of inner $1$-cocycles is thus isomorphic with $L^2(\cA) \cdot p_\eps$ and so, has $\cA''$-dimension equal to $1$.
\end{proof}

\section{\boldmath Derivations on $\Rep(A_u(F))$} \label{sec.derivations-Au}

In this section, we again specialize to the case of free unitary quantum groups. Let $F\in\mathrm{GL}_m(\CC)$. The methods of the previous section allow for a direct and explicit proof that $\beta^{(2)}_1(\Rep(A_u(F)))=1$. More generally, we determine the first cohomology of $\Rep(A_u(F))$ with arbitrary coefficients.

By \cite{banica-auf}, the category $\Rep(A_u(F))$ is freely generated by the fundamental representation $u$ and the irreducible representations can be labeled by words in $u$ and $\ubar$. To avoid confusion between words and tensor products, we explicitly write $\otimes$ to denote the tensor product of two representations. The tensor product $\ubar \ot u$ decomposes as the sum of the trivial representation $\eps$ and the irreducible representation with label $\ubar u$. Similarly, $u \ot \ubar \cong \eps \oplus u \ubar$. Moreover, the standard solutions of the conjugate equations for $u$, given by $t_u \in (\ubar \ot u,\eps)$ and $s_u \in (u \ot \ubar,\eps)$ generate all intertwiners between tensor products of $u$ and $\ubar$.

\begin{proposition}
\label{thm:free-cocycle-characterisation}
Let $F \in \GL_m(\C)$ and $\mathcal{C}=\Rep(A_u(F))$, with tube algebra $\cA$. Let $\cK$ be any nondegenerate right Hilbert $\cA$-module. Using the notation of Proposition \ref{prop:cocycle-upper-bound}, we find an isomorphism
\begin{equation}
\label{eqn:rep-au-cocycle-space}
Z^1(\Rep(A_u(F)),\mathcal{K}^0)\cong \cK \cdot p_\eps \oplus \cKfix_{\ubar u} \; .
\end{equation}
\end{proposition}

\begin{proof}
As explained in Section \ref{sec:representation-theory}, we consider $\mathcal{K}$ as the nondegenerate right Hilbert $\mathcal{A}$-module given by a unitary half braiding $\sigma$ on some ind-object $X$.
A vector on the right-hand side of \eqref{eqn:rep-au-cocycle-space} then corresponds to an element in $(u\otimes X,u)$ satisfying the conditions of Lemma \ref{lem:cocycle-necessary-conditions}, by Frobenius reciprocity.
Fix such a morphism $D_u\in (u\otimes X,u)$.
We have to show that $D_u$ comes from a $1$-cocycle $D\in \Hom_{\mathcal{B}}(p_\eps\cdot\mathcal{A},\mathcal{K}^0)$, which we will construct as a family of morphisms $(D_\alpha)_{\alpha\in\mathcal{C}}$ satisfying the naturality condition \eqref{eqn:dalpha-naturality}.
The identity \eqref{eqn:dalpha-adjoint-constraint} forces us to define $D_{\overline{u}}$ by
\[
D_{\overline{u}}=-\sigma^*_{\overline{u}}(t_u^*\otimes 1 \ot 1)(1 \otimes D_u\otimes 1)(1 \otimes s_u) \; .
\]
This is unambiguous because $u\neq\overline{u}$ in $\Rep(A_u(F))$. By Lemma \ref{lem:cocycle-necessary-conditions}, we also have that
\[
D_{\overline{u}}=-(1 \otimes s_u^*\otimes 1)(1 \ot 1 \otimes\sigma_{\overline{u}}^*)(1 \otimes D_u\otimes 1)(t_u\otimes 1) \; .
\]
The cocycle identity \eqref{eqn:cocycle-dalpha-charact} imposes the definition
\begin{equation}
\label{eqn:rep-au-induced-cocycle-products}
D_{\alpha_1\otimes\cdots\otimes\alpha_n}=\sum_{k=1}^n (1^{\otimes k}\otimes\sigma_{\alpha_{k+1}\otimes\cdots\otimes\alpha_n}^*)(1^{\otimes (k-1)}\otimes D_{\alpha_k}\otimes 1^{\otimes (n-k)})
\end{equation}
where $\alpha_k\in\{u,\overline{u}\}$. We also must set $D_\eps=0$. Since every irreducible object in $\Rep(A_u(F))$ is a subobject of some tensor product of $u$ and $\overline{u}$, these relations fix $D_\alpha$ for all $\alpha\in\mathcal{C}$. Concretely, if $\alpha\in\Irr(\mathcal{C})$ and $w:\alpha_1\otimes\cdots\otimes\alpha_n\to \alpha$ is a co-isometry, where $\alpha_k\in\{u,\overline{u}\}$, we set
\begin{equation}
\label{eqn:rep-au-induced-cocycle-restriction}
D_\alpha=(w\otimes 1)D_{\alpha_1\otimes\cdots\otimes\alpha_n}w^*.
\end{equation}
Now, since $\alpha$ appears in the decomposition of several different tensor products, it is not immediately clear why this is well defined.
To this end, we will show that the naturality relation
\begin{equation}
\label{eqn:rep-au-induced-cocycle-naturality}
(V\otimes 1)D_{\alpha_1\otimes\cdots\otimes\alpha_n}=D_{\alpha_1'\otimes\cdots\otimes\alpha_m'}V
\end{equation}
holds for all morphisms $V:\alpha_1\otimes\cdots\otimes\alpha_n\to\alpha_1'\otimes\cdots\otimes\alpha_m'$, with $\alpha_i,\alpha_j'\in\{u,\overline{u}\}$.
This is where the freeness of $\mathcal{C}$ comes into play.
By \cite[Lemme~6]{banica-auf}, the intertwiner spaces between tensor products involving $u$ and $\overline{u}$ are generated by maps of the form $1^{\otimes i}\otimes s_u\otimes 1^{\otimes j}$ and $1^{\otimes i}\otimes t_u\otimes 1^{\otimes j}$, and their adjoints.
Appealing to the naturality of $\sigma$ in \eqref{eqn:rep-au-induced-cocycle-products}, it is therefore sufficient to verify that
\begin{align*}
&D_{u\otimes\overline{u}}s_u=0 &(s_u^*\otimes 1)D_{u\otimes\overline{u}}=0 \; ,\\
&D_{\overline{u}\otimes u}t_u=0 &(t_u^*\otimes 1)D_{\overline{u}\otimes u}=0 \; ,
\end{align*}
which follows from the two different expressions for $D_{\overline{u}}$, by retracing the computations made in the proof of Proposition \ref{prop:cocycle-dalpha-properties}. We conclude that there exists a unique $D\in\Hom_{\mathcal{B}}(p_\eps\cdot \mathcal{A},\mathcal{K}^0)$ producing the family of maps $(D_\alpha)_{\alpha\in\mathcal{C}}$.
This family satisfies the cocycle relation \eqref{eqn:cocycle-dalpha-charact} by construction.
Therefore $D$ is a 1-cocycle, as required.
\end{proof}

Combining Propositions \ref{prop:cocycle-upper-bound} and \ref{thm:free-cocycle-characterisation}, we get
\[
\beta^{(2)}_1(\Rep(A_u(F)))=\tau(q_{\overline{u}u}) \; .
\]
Calculating $\beta^{(2)}_1(\Rep(A_u(F)))$ therefore boils down to computing the trace of $q_{\overline{u}u}$. By von Neumann's mean ergodic theorem, we have that
\begin{equation}
\label{eqn:u-moments-dim}
\lim_{n\to\infty} \frac{1}{n}\sum_{k=0}^{n-1} \tau(U_{\overline{u}u}^k)=\tau(q_{\overline{u}u}).
\end{equation}
In other words, to find the first $L^2$-Betti number of $\Rep(A_u(F))$, it is now sufficient to compute the traces $\tau(U_{\overline{u}u}^k)$ for all $k\in\NN$, i.e.\@ the sequence of moments of $U_{\overline{u}u}$.

The following lemma translates this problem into a combinatorial one.
\begin{lemma}
\label{lem:u-moments-zeta}
Let $\mathcal{C}$ be an arbitrary rigid $C^*$-tensor category with tube algebra $\mathcal{A}$.
For $\alpha\in \mathcal{C}$ and $k\geq 1$, define the rotation map
\[
\zeta_\alpha^k: (\alpha^k,\eps)\to (\alpha^k,\eps):
\xi\mapsto (1^{\otimes k} \otimes s_\alpha^*)(1 \otimes\xi\otimes 1)s_\alpha
\]
Then $\tau(U_i^k)=\overline{\Tr(\zeta_i^k)}$ for all $i\in\Irr(\mathcal{C})$, where $\Tr$ is the unnormalized trace on the finite-dimensional matrix algebra of linear transformations of $(i^k,\eps)$.
\end{lemma}

\begin{proof}
Fix $i\in\Irr(\mathcal{C})$ and observe that
\begin{align*}
\tau(U_i^k) &= \sum_{W\in \onb(i^k,\,\eps)} \Tr_i\left((1 \otimes W^*)(W\otimes 1)\right) =\sum_{W\in\onb(i^k,\,\eps)} s_i^*(1 \otimes W^*\otimes 1)(W\otimes 1 \ot 1)s_i\\
&=\sum_{W\in\onb(i^k,\,\eps)} s_i^*(1 \otimes W^*\otimes 1)(1^{\otimes k} \otimes s_i)W =\sum_{W\in\onb(i^k,\,\eps)} \langle W,\zeta_i^k(W)\rangle =\overline{\Tr(\zeta_i^k)} \; .
\end{align*}
\end{proof}

\begin{proposition}\label{thm:u-moments-aof}
Consider either $\Rep(A_u(F))$ for $F \in \GL_m(\C)$ or $\Rep(A_o(F))$ for $F \in \GL_m(\C)$ with $F \overline{F} = \pm 1$. In both cases, denote by $u$ the fundamental representation and let $\pi$ be the nontrivial irreducible summand of $u\otimes\overline{u}$ or $\overline{u}\otimes u$. Then, for all $k \in \Z$, we have
\begin{equation}
\label{eqn:freeqg-2-moments}
\tau(U_{\pi}^k)=\begin{cases}
d(u)^2-1 & \;\;\text{if}\; k = 0 \; ,\\
0 & \;\;\text{if}\; |k| = 1 \; \\
1 & \;\;\text{if}\; |k| \geq 2 \; .
\end{cases}
\end{equation}
So, $\tau(q_\pi) = 1$ and the spectral measure of $U_\pi$ with respect to the (unnormalized) trace $\tau$ on $p_\pi \cdot \cA'' \cdot p_\pi$ is given by $\delta_1 + (d(u)^2 - 2 - 2 \Re(z))\, dz$, where $dz$ denotes the normalized Lebesgue measure on the unit circle $S^1$.
\end{proposition}

\begin{proof}
We first deduce the result for $A_u(F)$ from the $A_o(F)$ case. Up to monoidal equivalence, we may assume that $F \overline{F} = \pm 1$. Consider the group $\ZZ$ as a $C^*$-tensor category with generator $z$, and denote the fundamental representation of $A_o(F)$ by $v$. Write $\pi$ for the nontrivial irreducible summand of $v\otimes v$. We can embed $\Rep(A_u(F))$ into the free product $\ZZ * \Rep(A_o(F))$ as a full subcategory, by sending the fundamental representation $u$ to $zv$, see \cite[Th\'{e}or\`{e}me~1(iv)]{banica-auf}. Under this identification, we have that $\overline{u}\otimes u = v\otimes v$, which implies that also $\overline{u}u=\pi$. By mapping $u$ to $vz$ instead, we similarly get that $u\otimes\overline{u}=v\otimes v$.

So it remains to prove the proposition for $\Rep(A_o(F))$, where $F\in \mathrm{GL}_m(\CC)$ with $m\geq 2$ and $F\overline{F}=\pm 1$.
If we choose $q\in [-1,1]\setminus \{0\}$ such that
\[
\Tr(F^*F)=|q|+|q|^{-1}\geq 2\qquad \text{ and }\qquad F\overline{F}=-\operatorname{sgn}(q)1 \; ,
\]
then it follows from \cite[Theorem~5.3]{bdrv-06} that $A_o(F)$ is monoidally equivalent to $\mathrm{SU}_q(2)$.
We still denote the fundamental representation by $v$.

Note that, strictly speaking, the category $\Rep(\mathrm{SU}_q(2))$ depends on the sign of $q$.
However, since we only work in the subcategory generated by $v\otimes v$, all parity issues disappear.
More precisely, if $v'$ denotes the fundamental representation of $\mathrm{SU}_{-q}(2)$, then the full $C^*$-tensor subcategories generated by $v\otimes v$ and $v'\otimes v'$ are monoidally equivalent. To see this, denote the Hopf $*$-subalgebra of $\Pol(\mathrm{SU}_q(2))$ generated by the matrix coefficients of $v\otimes v$ by $B$.
It suffices to remark that in the same way as in \cite[Corollary~4.1]{banica-as}, the adjoint coaction of $\mathrm{SU}_q(2)$ on $M_2(\CC)$ identifies $B$ with the quantum automorphism group of $(M_2(\CC),\phi_{q^2})$, where the state $\phi_{q^2}$ is given by
\[
\phi_{q^2}: M_2(\CC)\to \CC:\left(\begin{matrix} a_{11} & a_{12} \\ a_{21} & a_{22}\end{matrix}\right)\mapsto \frac{1}{1+q^2}\left(a_{11}+q^2 a_{22}\right) \; .
\]
In particular, the isomorphism class of $B$ does not depend on the sign of $q$. By duality, the full $C^*$-tensor subcategory of $\Rep(\mathrm{SU}_q(2))$ generated by $v\otimes v$ is also independent of the sign of $q$. We may therefore assume that $q<0$ without loss of generality. Since $U_\pi$ is unitary, it suffices to compute $\tau(U_\pi^k)$ for all positive integers $k$.

In summary, we have reduced the problem to a question about the Temperley-Lieb-Jones category $\mathcal{TL}_{d,-1}$, where $d=|q|+|q|^{-1}$ (cf.\ \cite[\S~2.5]{neshveyev-tuset}).
This category admits a well-behaved diagram calculus, see e.g.\ \cite{banica-speicher}.
In this view, morphisms from $v^{\otimes n}$ to $v^{\otimes m}$ are given by linear combinations of non-crossing pair partitions $p\in NC_2(n,m)$, which we will represent by diagrams of the following form:
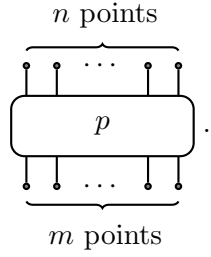
\[
\begin{tikzpicture}[baseline={([yshift=-0.5ex]current bounding box.center)},thick, scale=0.2, every circle node/.style={draw, fill=black!50, minimum width=2pt, inner sep=0pt}]
	 				\roundedbox{(5,4)}{2}{4}{$p$}
						\node[circle] (u-1) at (0,8) [draw] {};
						\node[circle] (u-2) at (2,8) [draw] {};
						\node[circle] (u-3) at (8,8) [draw] {};
						\node[circle] (u-4) at (10,8) [draw] {};
						\node at (5,8) {$\cdots$};
						\draw (u-1) -- (0,6);
						\draw (u-2) -- (2,6);
						\draw (u-3) -- (8,6);
						\draw (u-4) -- (10,6);
						\draw [decorate, decoration = {brace}] (0,9) -- (10,9)
						node [anchor=south,xshift=-2.5em,yshift=1ex] {$n$ points};
						\node[circle] (l-1) at (0,0) [draw] {};
						\node[circle] (l-2) at (2,0) [draw] {};
						\node[circle] (l-3) at (8,0) [draw] {};
						\node[circle] (l-4) at (10,0) [draw] {};
						\node at (5,0) {$\cdots$};
						\draw (l-1) -- (0,2);
						\draw (l-2) -- (2,2);
						\draw (l-3) -- (8,2);
						\draw (l-4) -- (10,2);
						\draw [decorate, decoration = {brace, mirror}] (0,-1) -- (10,-1)
						node [anchor=north,xshift=-2.5em,yshift=-1ex] {$m$ points};
					\end{tikzpicture} \; .
		\]
The composition $pq$ of diagrams $p$ and $q$, whenever meaningful, is defined by vertical concatenation, removing any loops that arise.
The tensor product and adjoint operations are given by horizontal concatenation and reflection along the horizontal axis, respectively.
We will denote the morphism in $(v^{\otimes m},v^{\otimes n})$ associated to the partition $p\in NC_2(n,m)$ by $T_p$.
One then has that $T_p=T_{p^*}$, $T_{p\otimes q}=T_p\otimes T_q$ and $T_pT_q=d^{\ell(p,q)}T_{pq}$, where $\ell(p,q)$ denotes the number of loops removed in the composition of $p$ and $q$.
Moreover, the family $\{T_p\mid p\in NC_2(n,m)\}$ is a basis for $(v^{\otimes m},v^{\otimes n})$.

In view of Lemma \ref{lem:u-moments-zeta}, we now specialize to non-crossing pair diagrams without upper points, i.e. morphisms in $((v\otimes v)^{\otimes k},\eps)$.
The action of $\zeta^k_{v\otimes v}$ (as defined in Lemma \ref{lem:u-moments-zeta}) on intertwiners of the form $T_p$ for $p\in NC_2(0,2k)$ has an easy description in terms of the partition calculus discussed above:
\[
\zeta^k_{v\otimes v}(T_p)=T_{\sigma_k(p)} \; ,
\]
where $\sigma_k$ is the permutation of $NC_2(0,2k)$ given by
\[
\sigma_k\left(
			\begin{tikzpicture}[baseline={([yshift=-0.5ex]current bounding box.center)},thick, scale=0.2, every circle node/.style={draw, fill=black!50, minimum width=2pt, inner sep=0pt}]
					\roundedbox{(5,4)}{2}{4}{$p$}
					\node[circle] (u-1) at (0,0) [draw] {};
					\node[circle] (u-2) at (2,0) [draw] {};
					\node[circle] (u-3) at (8,0) [draw] {};
					\node[circle] (u-4) at (10,0) [draw] {};
					\node at (5,0) {$\cdots$};
					\draw (u-1) -- (0,2);
					\draw (u-2) -- (2,2);
					\draw (u-3) -- (8,2);
					\draw (u-4) -- (10,2);
				\end{tikzpicture}\right)
			=
			\begin{tikzpicture}[baseline={([yshift=-0.85ex]current bounding box.center)},thick, scale=0.2, every circle node/.style={draw, fill=black!50, minimum width=2pt, inner sep=0pt}]
					\roundedbox{(5,4)}{2}{4}{$p$}
					\node[circle] (u-3new) at (-4,0) [draw] {};
					\node[circle] (u-4new) at (-2,0) [draw] {};
					\node[circle] (u-1) at (0,0) [draw] {};
					\node[circle] (u-2) at (2,0) [draw] {};
					\coordinate (u-3old) at (8,-1) [draw] {};
					\coordinate (u-4old) at (10,0) [draw] {};
					\coordinate (u-3b) at (8,1);
					\coordinate (u-4b) at (10,2);
					\node at (5,0) {$\cdots$};
					\draw (u-1) -- (0,2);
					\draw (u-2) -- (2,2);
					\draw (10,2) -- (u-4old) -- (12,0) -- (12,7) -- (-2,7) -- (u-4new);
					\draw (8,2) -- (u-3old) -- (13,-1) -- (13,8) -- (-4,8) -- (u-3new);
				\end{tikzpicture} \; .
		\]
In other words, $\zeta^k_{v\otimes v}$ permutes a basis of $((v\otimes v)^{\otimes k},\eps)$.
In fact, $\zeta_\pi^k$ behaves similarly with respect to a suitable basis of $(\pi^{\otimes k},\eps)$.
Let $Q:v\otimes v\to \pi$ be a co-isometry.
We proceed to argue that the intertwiners $\{Q^{\otimes k}T_p\mid p\in NC_2^\circ(k)\}$ form a basis of $(\pi^{\otimes k},\eps)$, where
\[
NC_2^\circ(k)=\{p\in NC_2(0,2k)\mid i\text{ odd}\implies\{i,i+1\}\notin p\} \; .
\]
Indeed, it is clear that multiplication by $Q^{\otimes k}$ yields a linear map from $((v\otimes v)^{\otimes k},\eps)$ to $(\pi^{\otimes k},\eps)$.
Moreover, it is easy to see that $T_p$ lies in the kernel of this map whenever $p\in NC_2(0,2k)\setminus NC_2^\circ(k)$.
Hence, to finish the proof of the claim, it suffices to check that
\begin{equation}
\label{eqn:dimension-expression}
\dim_{\CC} (\pi^{\otimes k},\eps) = \# NC_2^\circ(k).
\end{equation}
This fact is probably well known, but we give a short proof here for completeness.
The number of elements of $NC_2^\circ(k)$ is known in the combinatorial literature as the $k$-th Riordan number.
As shown in \cite[\S~3.2 (R2), \S~5]{bernhart-1999}, the Riordan numbers can be expressed in terms of the Catalan numbers $C_i$ by means of the formula
\begin{equation}
\label{eqn:riordan-identity}
\# NC_2^\circ(k)=\sum_{i=0}^k (-1)^{k-i}\binom{k}{i} C_i.
\end{equation}

The left-hand side of \eqref{eqn:dimension-expression} only depends on the fusion rules of the tensor powers of $\pi$, so we can take $\pi$ to be the three-dimensional irreducible representation of $SU(2)$ for the purposes of this part of the computation.
Making use of the Weyl integration formula for $SU(2)$, see \cite[Example~11.33]{hall-lie-reps}, we find that
\begin{align*}
\label{eqn:riordan-integral}
\dim_{\CC}(\pi^{\otimes k},\eps)&=\int_{SU(2)}\chi_\pi^k(g)\diff g=\frac{4}{\pi}\int_0^{\pi/2} (4\cos^2\theta-1)^k\sin^2\theta\diff\theta\\
&=\frac{4}{\pi}\int_0^1(4x^2-1)^k\sqrt{1-x^2}\diff x\\
&=\frac{4}{\pi}\sum_{i=0}^k 4^i (-1)^{k-i}\binom{k}{i} \int_0^1 x^{2i}\sqrt{1-x^2} \diff x \; .
\end{align*}
By the moment formula for the Wigner semicircle distribution, this is precisely \eqref{eqn:riordan-identity}.

Having shown that the intertwiners of the form $Q^{\otimes k}T_p$ form a basis of $(\pi^{\otimes k},\eps)$, we now demonstrate that $\zeta_\pi^k$ acts on this basis by permutation.
To this end, observe that $s_\pi=(Q\otimes Q)s_{v\otimes v}$.
For $\xi\in ((v\otimes v)^{\otimes k},\eps)$, this yields
\begin{align*}
\zeta^k_{\pi}(Q^{\otimes k}\xi)
&=Q^{\otimes k}(1^{\otimes 2k} \otimes s_\pi^*)(1 \otimes\xi\otimes 1)s_\pi\\
&=Q^{\otimes k}(1^{\otimes 2k}\otimes s_{v\otimes v}^*)(1^{\otimes 2k}\otimes Q^*Q\otimes Q^*Q)(1^{\otimes 2}\otimes \xi\otimes 1^{\otimes 2})s_{v\otimes v}\\
&=Q^{\otimes k}\zeta^k_{v\otimes v}(\xi) \; .
\end{align*}
where the last equality follows by substituting $Q^*Q=1-d(v)^{-1}s_vs_v^*$ and noting that all terms involving $s_vs_v^*$ vanish.
In summary,
\[
\zeta_\pi^k(Q^{\otimes k}T_p)=Q^{\otimes k}T_{\sigma_k(p)}
\]
for all $p\in NC_2^\circ(k)$. So $\zeta^k_\pi$ permutes a basis of $(\pi^{\otimes k},\eps)$, as claimed.
It follows that the trace of $\zeta^k_\pi$ is exactly the number of fixed points of $\sigma_k$ that lie in $NC_2^\circ(k)$.
When $k=1$, this set is empty, but for all $k\geq 2$ there is a unique such fixed point, given by the partition
\[
\begin{tikzpicture}[baseline={([yshift=1.2ex]current bounding box.center)},thick, scale=0.2, every circle node/.style={draw, fill=black!50, minimum width=2pt, inner sep=0pt}]

				\node[circle] (outer1) at (0,0) [draw] {};
				\node[circle] (pair1-1) at (2,0) [draw] {};
				\node[circle] (pair1-2) at (4,0) [draw] {};
				\node[circle] (pair2-1) at (10,0) [draw] {};
				\node[circle] (pair2-2) at (12,0) [draw] {};
				\node[circle] (outer2) at (14,0) [draw] {};
				\node (dots) at (7,0) {$\cdots$};
				\coordinate (outer1b) at (0,2);
				\coordinate (outer2b) at (14,2);
				\coordinate (pair1-1b) at (2,1);
				\coordinate (pair1-2b) at (4,1);
				\coordinate (pair2-1b) at (10,1);
				\coordinate (pair2-2b) at (12,1);
				\coordinate (brace1) at (2,-1);
				\coordinate (brace2) at (12,-1);
				\draw (outer1) -- (outer1b) -- (outer2b) -- (outer2);
				\draw (pair1-1) -- (pair1-1b) -- (pair1-2b) -- (pair1-2);
				\draw (pair2-1) -- (pair2-1b) -- (pair2-2b) -- (pair2-2);
				\draw [decorate, decoration = {brace, mirror}] (brace1) -- (brace2)
					node [anchor=north,xshift=-2.5em,yshift=-1ex] {$k-1$ pairs};
			\end{tikzpicture} \; .
		\]
Since
$$\lim_{n\to\infty} \frac{1}{n}\sum_{k=0}^{n-1} \tau(U_\pi^k) = 1 \; ,$$
we conclude that $\tau(q_\pi) = 1$. Clearly, the measure on $S^1$ in the formulation of the proposition has the same moments as $U_\pi$ and thus is the spectral measure of $U_\pi$.
\end{proof}

\begin{remark}
From the computation for $A_o(F)$, one might be tempted to conjecture that the trace of the spectral projection $q_i$ is always less than 1 for all $i\in\Irr(\mathcal{C})$ in any $C^*$-tensor category.
However, this is not the case.
Consider the category of finite-dimensional unitary representations of the alternating group $A_4$.
This category has four equivalence classes of irreducible objects, which we will denote by $\eps, \omega_1,\omega_2$ and $\pi$.
The trivial representation corresponds to $\eps$, and $\omega_1,\omega_2$ are one-dimensional representations that can be thought of as ``cube roots of $\eps$'', in that $\omega_1=\overline{\omega_2}$, and $\omega_1\otimes\omega_1=\omega_2$.
The remaining representation $\pi$ is 3-dimensional, and satisfies
\[
\pi\otimes \pi\cong\eps\oplus\omega_1\oplus\omega_2\oplus\pi\oplus\pi \; .
\]
Fix a partition of the identity into pairwise orthogonal projections
\[
1_{\pi\otimes\pi}=P_\eps+P_{\omega_1\oplus\omega_2}+P_{\pi\oplus\pi}
\]
such that the image of $P_\alpha$ is isomorphic to $\alpha$.
Using numerical methods, we found that
\begin{align*}
q_\pi &=\frac{7}{18}p_\pi\oplus \frac{1}{18} 1_\pi\oplus\frac{1}{18}1_\pi\oplus\left(\frac{7}{6}P_\eps+\frac{1}{6}P_{\omega_1\oplus\omega_2}+\frac{1}{3}P_{\pi\oplus\pi}\right)\\
&\qquad \in (\pi\eps,\eps\pi)\oplus(\pi\omega_1,\omega_1\pi)\oplus(\pi\omega_2,\omega_2\pi)\oplus (\pi\pi,\pi\pi)=p_\pi\cdot \mathcal{A}\cdot p_\pi \; .
\end{align*}
In particular,
\[
\tau(q_\pi)=d(\pi)7/18=7/6 > 1 \; .
\]
\end{remark}

\end{document}